\pdfoutput=1
\documentclass[a4paper]{amsart}

\usepackage{enumitem}

\usepackage[usenames,dvipsnames]{xcolor}
\usepackage{pgfplots}
\usetikzlibrary{math}
\usepackage{tikz}
\usepackage{forest}

\usepackage[utf8]{inputenc}
\usepackage{dsfont}
\usepackage{amssymb,amsthm,amsmath}
\usepackage{mathrsfs,mathtools}
\usepackage{microtype}

\usepackage[style=authoryear,backend=bibtex,style=alphabetic,  backref=false, giveninits=true,citestyle=alphabetic, natbib=false, url=false, isbn=false,doi=true,eprint=false,clearlang=true,maxbibnames=99]{biblatex}
\DeclareRedundantLanguages{english,German,french}{english,german,french}
\DeclareFieldFormat[article]{title}{{#1}}
\DeclareFieldFormat[unpublished]{title}{{#1}}
\DeclareFieldFormat[proceedings]{title}{{#1}}
\DeclareFieldFormat[incollection]{title}{{#1}}
\DeclareFieldFormat[unpublished]{title}{{\em #1}}
\renewbibmacro{in:}{}
\DeclareFieldFormat[article]{pages}{#1}

\usepackage[colorlinks=true,linkcolor=Red,citecolor=Green]{hyperref}

\theoremstyle{plain}
\newtheorem{prop}{Proposition}[section]
\newtheorem{lemma}[prop]{Lemma}

\newtheorem{thm}[prop]{Theorem}

\newtheorem{cor}[prop]{Corollary}

\theoremstyle{definition}

\newtheorem{defi}[prop]{Definition}

\theoremstyle{remark}

\newtheorem*{ack}{Acknowledgement}

\makeatletter
\@namedef{subjclassname@2020}{%
  \textup{2020} Mathematics Subject Classification}
\makeatother

\newcommand{\masterint}{X}
\DeclareMathOperator{\RF}{RF} 

\DeclareMathOperator*{\starlim}{\star-lim}

\newcommand{\leb}{m} 
\newcommand{\FM}{F} 
\newcommand{\ME}{T} 
\newcommand{\cw}{\lambda} 
\newcommand{\inter}{I} 
\newcommand{\prt}{\xi} 
\renewcommand{\emptyset}{\varnothing} 
\renewcommand{\tilde}{\widetilde}

\newcommand{\one}{\mathds{1}}

\newcommand{\dd}{\:\mathrm{d}}

\newcommand{\TO}{\mc L}
\newcommand{\PF}{\mc P}

\newcommand{\HH}{\mathbb{H}}

\DeclareMathOperator{\GL}{GL}

\DeclareMathOperator{\SL}{SL}
\DeclareMathOperator{\PSL}{PSL}
\DeclareMathOperator{\PGL}{PGL}

\DeclareMathOperator{\Stab}{Stab}

\newcommand{\st}{\text{st}}

\newcommand\N{\mathbb{N}}
\newcommand\Q{\mathbb{Q}}
\newcommand\R{\mathbb{R}}
\newcommand\Z{\mathbb{Z}}
\newcommand\C{\mathbb{C}}

\newcommand{\h}{\mathbb{H}}

\newcommand{\mc}[1]{\mathcal #1}

\newcommand{\wt}{\widetilde}
\newcommand{\wh}{\widehat}

\newcommand{\eps}{\varepsilon}

\DeclareMathOperator{\dist}{dist}

\DeclareMathOperator{\id}{id}

\DeclareMathOperator{\Fct}{Fct}

\newcommand{\bmat}[4]{\begin{bmatrix} #1&#2\\#3&#4\end{bmatrix}}

\newcommand{\textbmat}[4]{\left[\begin{smallmatrix} #1&#2 \\ #3&#4
\end{smallmatrix}\right]}

\begin{document}

\title[Equidistribution of cusp points]{Equidistribution of cusp points\\ of  Hecke triangle groups}

\author{Laura Breitkopf}
\author{Marc Kesseb\"ohmer}
\author{Anke Pohl}
\address{University of Bremen, Department~3
-- Mathematics, Institute for Dynamical Systems, Bibliothekstr.~5, 28359 Bremen, Germany}
\email{\{breitlau, mhk, apohl\}@uni-bremen.de}

\begin{abstract} 
In the framework of infinite ergodic theory, we derive equidistribution results for suitable weighted sequences of cusp points of Hecke triangle groups encoded by group elements of constant word length with respect to a set of natural generators.  This is a generalization of the corresponding results for the modular group, for which we rely on advanced results from infinite ergodic theory and transfer operator techniques  developed for AFN-maps.
\end{abstract}

\maketitle

\tableofcontents

\section{Introduction and statement of main results}\label{sec:intro}

The classical Stern--Brocot sequence $(\mc S_n)_{n\in\N_0}$ partitions the rational numbers in the interval~$[0,1]$ into subsets by sorting them by means of iterated mediants starting from the two seminal reduced fractions~$0/1$ and~$1/1$. Here, a \emph{mediant} of two reduced fractions~$a/b$ and~$c/d$ means the fraction~$(a+c)/(b+d)$. In other words, with a view to the (modified) Stern--Brocot tree of iterated mediants as indicated in Figure~\ref{fig:SBtree}, the set $\mc S_n$ consists of the (reduced) fractions that are new at tree level~$n$. Thus
\[
\mc S_0 = \left\{ \frac01, \frac11 \right\}\,,\quad \mc S_1 = \left\{\frac12\right\}\,,\quad \mc S_2 = \left\{\frac13, \frac23\right\}\,,\quad \mc S_3 = \left\{ \frac14, \frac25, \frac35, \frac34\right\}\,,\quad \ldots\,.
\]
\begin{figure}
\centering
\begin{forest}
  Stern Brocot/.style n args={5}{%
    content=$\frac{\number\numexpr#1+#3\relax}{\number\numexpr#2+#4\relax}$,
    if={#5>0}{
      append={[,Stern Brocot={#1}{#2}{#1+#3}{#2+#4}{#5-1}]},
      append={[,Stern Brocot={#1+#3}{#2+#4}{#3}{#4}{#5-1}]}
    }{}}
[,Stern Brocot={0}{1}{1}{1}{3}]
\end{forest}
\caption{The  Stern--Brocot tree restricted to $[0,1]$, starting at tree level~$1$.}
\label{fig:SBtree}
\end{figure}
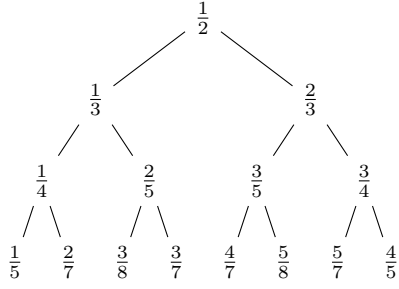

The number-theoretical sequence of the Stern--Brocot sets~$(\mc S_n)_{n\in\N_0}$ is well studied and subject of a correspondingly vast body of literature. See, e.g., \cite{MR2854637,MR2864376,MR3247168,MR3809715,MR3976469,MR4102624,MR4430102}.  Also its  connection to Minkowski's question mark function has   gained some interest in the last decades \cite{MR2444218,MR3319502,MR3411226,MR3692125,MR3709533,MR3977887,MR4221205}.  Starting from intensive investigations in the theoretical physics literature, the closely connected \emph{Farey fractional spin chain} has been an important model for thermodynamic systems with first and second order phase transitions showing also intermittent behavior as a dynamical system  \cite{MR1700150,MR1966324,MR2096044,MR2227090,MR2594905,MR2773858,MR4661063}.
Enumerating the elements of $\mc S_n$ for $n\geq1$ in increasing order, say
\[
\mc S_{n}=\{r^{n}_{\ell}:\ell=1,\ldots, 2^{n-1}\}\,,
\]
allows us to define the family of \emph{even Stern--Brocot intervals}
\[
E_{n}\coloneqq\left\{\left[r^{n+1}_{2\ell-1},r^{n+1}_{2\ell}\right] : \ell=1,\ldots,2^{n-1} \right\}\,,\quad n\in \N\,,
\]
e.g.,
\[
E_{1}=\left\{\left[\frac13,\frac23\right]\right\}\,,\quad E_{2}=\left\{\left[\frac14, \frac25\right],\ \left[\frac35, \frac34\right]\right\}\,,\quad\ldots.
\]
At the center of the investigations of the Farey fractional spin chain is the \emph{even Stern--Brocot partition function}
\[
Z_{n}(t)\coloneqq \sum _{I\in E_{n}}|I|^{t}\,,
\]
which was first given in this form by Feigenbaum, Procaccia and T\'el in~\cite{MR0998529}, and the corresponding pressure function
\[
P(t)\coloneqq \limsup_{n\to \infty }\frac1n\log Z_{n}(t)\,.
\]
Here, $|I|$ denotes the length of the interval~$I\in E_n$.
For the connection of~$Z_n$ and $P$ to the topological pressure of the Farey map on~$[0,1]$ (for a definition see Section~\ref{sec:mainresults}, case $q=3$) with geometric potential, the connection to the multifractal Lyapunov spectrum, and a detailed study of the phase transition we refer to~\cite{MR2069368,MR2227090, MR2338129}. From \cite{MR1966324, MR2338129} it follows that $P$ has a phase transition of the second kind in $1$, i.e., the pressure function $P$ is differentiable in $1$ with derivative $0$, but the second derivative of $P$ has a discontinuity there. From the multifractal viewpoint this means that  with respect to the Farey map there exists no invariant \emph{finite} Gibbs measure with Hausdorff dimension $1$ and Lyapunov exponent $0$. In fact, there exists an \emph{infinite} absolutely continuous invariant measure with density $x\mapsto 1/x$. By this observation we enter the domain of infinite ergodic theory, which will also be central to this paper.
To better understand the nature of the phase transition of~$P$ in~$1$, Kleban and Fiala in \cite{MR2096044} posed the question of the asymptotic behavior of~$Z_{n}(1)$ for $n\to\infty$. It turned out that the complete answer to this question lies at the center of infinite ergodic theory \cite{MR2900554,MR3052943,MR3585883,MR3459025}, which led to (equi-)distribution results that are also the subject of this paper and which are discussed in more detail in Section~\ref{sec:mainresults}.

As already indicated above,  the Stern--Brocot sequence is intimately related to the Farey map and the geodesic flow on the modular surface~$\PSL_2(\Z)\backslash\HH$. Here, $\HH$ refers to the hyperbolic plane. This relation between the \textit{a priori} arithmetic Stern--Brocot sequence and the dynamics of the geodesic flow triggers the question as to which degree results on the Stern--Brocot sequence, such as the distribution results alluded to above, are intrinsically of arithmetic nature (i.e., are inherent to the arithmeticity of~$\PSL_2(\Z)$) or generalize to non-arithmetic but somewhat similar settings.

With this article we contribute to this question in the following way: For the family of cofinite Hecke triangle groups~$\Gamma_q$ with odd angle parameter~$q$ (see Sections~\ref{sec:mainresults} and~\ref{sec:fundpropFarey} for details), we consider generalizations of the Farey map and the Stern--Brocot sequence that arise from a discretization of the geodesic flow on the Hecke triangle orbisurface~$\Gamma_q\backslash\HH$. As $\Gamma_3 = \PSL_2(\Z)$, and $\Gamma_q$ is non-arithmetic for $q\geq 5$, we have here a one-parameter sequence of Fuchsian groups that proceeds immediately from our initial, arithmetic setting of~$\PSL_2(\Z)$ into a non-arithmetic regime while remaining close to each other from a dynamical point of view. We study certain (equi-)distribution results of the generalized Stern--Brocot sequences and further generalizations as presented in Section~\ref{sec:mainresults} below and observe that our findings generalize the corresponding distribution results for the classical, arithmetic Stern--Brocot sequence (and its generalizations within the modular group~$\Gamma_3 = \PSL_2(\Z)$) without any limitations. Thus, for this type of results, the arithmeticity of~$\PSL_2(\Z)$ is not relevant; these are purely dynamical results.

\subsection{Main results}\label{sec:mainresults}

The Hecke triangle group~$\Gamma_q$ with angle parameter~$q\in\N_{\geq 3}$, or cusp width $\cw_q \coloneqq 2\cos(\pi/q)$, is the subgroup of~$\PSL_2(\R)$ generated by the two elements
\[
 \ME_q \coloneqq \bmat{1}{\cw_q}{0}{1}\quad\text{and}\quad S \coloneqq \bmat{0}{1}{-1}{0}\,.
\]
Here, and throughout this article, we denote an element in~$\PSL_2(\R)$ (or, more generally, in~$\PGL_2(\R)$) by a representing matrix in~$\SL_2(\R)$ (more generally, in~$\GL_2(\R)$) but with square brackets. From now on, we restrict the considerations to odd values of~$q$.

In order to state the generalized Farey map~$\FM_q$ associated with~$\Gamma_q$ in explicit terms, we set
\[
 U_q \coloneqq \ME_q S = \bmat{\cw_q}{-1}{1}{0}
\]
and, for any $k\in\Z$, 
\begin{equation}\label{eq:gk}
g_{q,k} \coloneqq \left(U_q^kS\right)^{-1} = \bmat{ s(k) }{ -s(k+1) }{ -s(k-1) }{ s(k) }
\end{equation}
with
\begin{equation*}
s(x) \coloneqq \frac{\sin\left(\frac{x}{q}\pi\right)}{\sin\left(\frac{\pi}q\right)}\,.
\end{equation*}
Since $U_q^q=\id$, the sequence $(g_{q,k})_{k\in\Z}$ is indeed periodic (with minimal period~$q$). Thus $g_{q,k+q} = g_{q,k}$ for all~$k\in\Z$. The element
\[
 Q \coloneqq \bmat{0}{1}{1}{0}
\]
of~$\PGL_2(\R)$ constitutes an outer symmetry of~$\Gamma_q$, i.e., $Q\Gamma_qQ=\Gamma_q$. In particular, $Qg_{q,k} = g_{q,q-k}Q$ for all~$k\in\Z$. The \emph{generalized Farey map}~$\FM_q$ is the selfmap on~$[0,1]$ that is piecewise given by the bijections
\begin{align}
\label{eq:FM1}
[g_{q,k}^{-1}.0, g_{q,k}^{-1}.1] &\to [0,1]\,,\quad x\mapsto g_{q,k}.x\,,
\shortintertext{and}
\label{eq:FM2}
[(Qg_{q,k})^{-1}.1, (Qg_{q,k})^{-1}.0]&\to [0,1]\,,\quad x\mapsto Qg_{q,k}.x\,,
\end{align}
with $k\in\{(q+1)/2,\ldots, q-1\}$. Here, the action of $\Gamma_q$ and $Q$ on~$P^1(\R) = \R\cup\{\infty\}$ is given by fractional linear transformation. Thus
\[
 \bmat{a}{b}{c}{d}.x = \frac{ax+b}{cx+d}
\]
with the mnemonic $1/0=\infty$ (i.e., continuous extension to $\infty$ and potential singularities), for any $x\in P^1(\R)$ and any element $\textbmat{a}{b}{c}{d}$ in the group generated by $\Gamma_q$ and $Q$.
As 
\begin{align*}
 (Qg_{q,k+1})^{-1}.0 & = g_{q,k+1}^{-1}.\infty = g_{q,k}^{-1}.0
 \intertext{and}
 (Qg_{q,k})^{-1}.1 & = g_{q,k}^{-1}.1
\end{align*}
for any $k\in\Z$, the map~$\FM_q$ is indeed well-defined. 
We note that $0$ is a fixed point of~$\FM_q$. 
We refer to Figure~\ref{fig:FM5} for an indication of the graph of~$\FM_q$. The \emph{generalized Stern--Brocot sequence}~$(\mc S_{n,q})_{n\in\N_0}$ is then defined by
\[
\mc S_{n,q}\coloneqq \wt{\mc S}_{n,q}\setminus \mc S_{n-1,q} \qquad\text{for $n\in\N_0$}\,,
\]
where
\[
\wt{\mc S}_{n,q} \coloneqq \FM_q^{-n}(0) \cup \FM_q^{-n}(1)\,,\quad \mc S_{-1,q} \coloneqq \emptyset\,.
\]
In Figure~\ref{fig:SB5}, an indication of the location of the first few elements for the generalized Stern--Brocot sequence for $q=5$ is shown.
For each odd $q\geq 3$, the union of Stern--Brocot elements is $\Gamma_q.\infty\cap[0,1]$. We have $\Gamma_3.\infty\cap[0,1]=\Q\cap[0,1]$, and $\Gamma_5.\infty\cap[0,1]=\Q(\sqrt{5})\cap[0,1]$. For some more values of $q$ this set is known explicitly, see \cite[Satz~2]{leutbecher_ueber_1974}.
In Theorems~\ref{thm:jac_weight_lim_intro} and~\ref{thm:weighted_distr} we will consider even more general sequences than Stern--Brocot; the Stern--Brocot sequence is related to the case $x=1$ in Theorem~\ref{thm:jac_weight_lim_intro} and $v/w=1$ in Theorem~\ref{thm:weighted_distr}.

\begin{figure}
	\centering
	 
\pgfmathsetmacro{\mylambda}{{1/2*(1+sqrt(5))}} 

\begin{tikzpicture}
	\pgfplotsset{width=.6\textwidth}
	\tikzmath{
		\c1 = {1/(\mylambda+1)};
		\c2 = {1/\mylambda};
		\c3 = {\mylambda/2};
	}
	\begin{axis}[xtick align=center, xmajorgrids, enlargelimits=false,
		ymin=0, ymax=1, ytick={0,1},
		xmin=0, xmax=1, xtick={0,\c1,\c2,\c3,1}, xtick pos = lower, xticklabels={0,$\frac{1}{\lambda+1}$,$\frac{1}{\lambda}$,$\frac{\lambda}{2}$,1},
		domain=0:1, samples=101
		]
		\addplot[domain=0:\c1] {x/(-\mylambda*x + 1)};
		\addplot[domain=\c1:\c2] {(1-\mylambda*x)/x};
		\addplot[domain=\c2:\c3] {(\mylambda*x - 1)/(-\mylambda*x + \mylambda)};
		\addplot[domain=\c3:1] {(-\mylambda*x+\mylambda)/(\mylambda*x - 1)};
	\end{axis}
\end{tikzpicture}
	\caption{Graph of the generalized Farey map associated to $q=5$}
	\label{fig:FM5}
\end{figure}
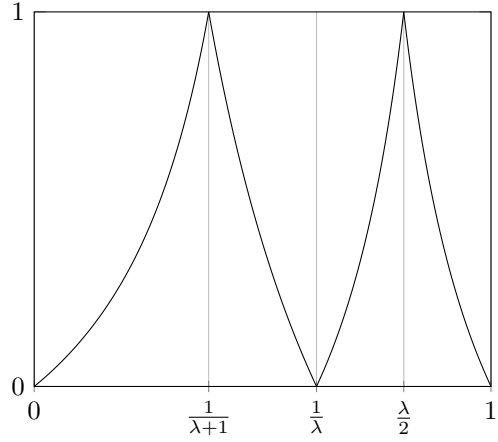

\begin{figure}
	\centering
	\begin{tikzpicture}[xscale=9, yscale=1, thick]
	\draw (0,0) node[left, xshift=-0.5cm]{$n=1$} -- (1,0);
	
	\def\tl{-0.2} 
	\def\dist{-1} 
	
	\draw (0,0) -- (0,\tl) node[below]{$0$};
	\draw (0.382,0) -- (0.382,\tl) node[below]{$\frac{1}{\lambda+1}$};
	\draw (0.618,0) -- (0.618,\tl) node[below]{$\frac{1}{\lambda}$};
	\draw (0.809,0) -- (0.809,\tl) node[below]{$\frac{\lambda}{2}$};
	\draw (1,0) -- (1,\tl) node[below]{$1$};

	
	\draw (0,\dist) node[left, xshift=-0.5cm]{$n=2$} -- (1,\dist);
	
	\draw (0,\dist) -- (0,\dist+\tl) node[below]{$0$};
	\draw ({0.236},\dist) -- ({0.236},\dist+\tl) node[below]{$\frac{1}{2\lambda +1}$};
	\foreach \x in {0.309, 0.35, 0.412, 0.447}{
		\draw (\x,\dist) -- (\x,\dist+\tl);
	}
	\draw ({0.5},\dist) -- ({0.5},\dist+\tl) node[below]{$\frac{1}{2}$};
	
	\draw ({0.723},\dist) -- ({0.723},\dist+\tl) node[below]{$\frac{2\lambda+1}{3\lambda+1}$};
	\foreach \x in {0.764, 0.789, 0.829, 0.854}{
		\draw (\x,\dist) -- (\x,\dist+\tl);
	}
	\draw ({0.894},\dist) -- ({0.894},\dist+\tl) node[below]{$\frac{2\lambda+2}{3\lambda+1}$};
	
	\draw (1,\dist) -- (1,\dist+\tl) node[below]{$1$};
	
	
	\draw (0,2*\dist) node[left, xshift=-0.5cm]{$n=3$} -- (1,2*\dist);
	
	\draw (0,2*\dist) -- (0,2*\dist+\tl) node[below]{$0$};
	\foreach \x in {0.171, 0.206, 0.224, 0.247, 0.259, 0.276, 0.333, 0.342, 0.347, 0.354, 0.359, 0.365}{
		\draw (\x,2*\dist) -- (\x,2*\dist+\tl);
	}
	
	\foreach \x in {0.398, 0.405, 0.409, 0.415, 0.42, 0.427, 0.472, 0.484, 0.493, 0.508, 0.519, 0.539}{
		\draw (\x,2*\dist) -- (\x,2*\dist+\tl);
	}
	
	\foreach \x in {0.691, 0.708, 0.717, 0.729, 0.736, 0.745, 0.778, 0.783, 0.786, 0.791, 0.794, 0.798}{
		\draw (\x,2*\dist) -- (\x,2*\dist+\tl);
	}
	
	\foreach \x in {0.82, 0.824, 0.827, 0.832, 0.835, 0.84, 0.873, 0.882, 0.889, 0.901, 0.91, 0.927}{
		\draw (\x,2*\dist) -- (\x,2*\dist+\tl);
	}
	
	\draw (1,2*\dist) -- (1,2*\dist+\tl) node[below]{$1$};
	
\end{tikzpicture}
	\caption{First elements of generalized Stern--Brocot sequence $(\mc S_{n,5})$, omitting $\mc S_{0,5} = \{0,1\}$}
	\label{fig:SB5}
\end{figure}
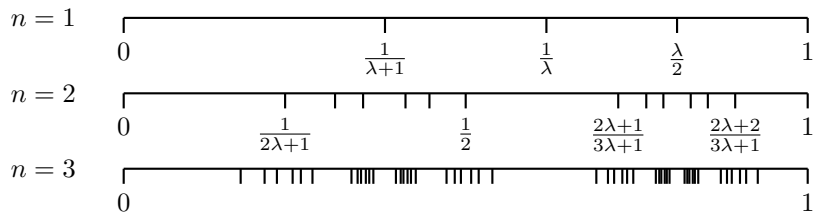

Our first main result concerns the shrinking properties of compact subsets of~$[0,1]$ that are bounded away from the fixed point~$0$ under backwards propagation via~$\FM_q$ with respect to the Lebesgue measure~$\leb$ on~$[0,1]$.
We endow $[0,1]$ with its Borel $\sigma$-algebra~$\mc B_{[0,1]}$, and we denote by $\starlim$ the weak-star limit of  measures on~$[0,1]$, i.e., a sequence of measures~$(\nu_n)_n$ on $[0,1]$ (or $\mc B_{[0,1]}$) converges to the measure~$\nu$ on~$[0,1]$ if and only if for all continuous functions $f$ on~$[0,1]$ we have
\[
 \nu_n(f) \to \nu(f)\,.
\]
We note that weak-star convergence of measures is sometimes called ``weak convergence of measures,'' omitting ``star.'' We further note that each continuous function on~$[0,1]$ is of course bounded due to the compactness of~$[0,1]$.

\begin{thm}\label{thm:dyn_intervals}
Let $q\geq 5$ odd. For all~$0<\alpha\leq\beta\leq 1$  we have
\[
 \starlim_{n\to\infty} \left( \log(n) \cdot \leb\vert_{\FM_q^{-n}\bigl( [\alpha,\beta] \bigr)}  \right) = \log\left({\beta}/{\alpha}\right) \leb\,.
\]
\end{thm}

An immediate consequence of Theorem~\ref{thm:dyn_intervals} is the observation that for $\alpha$ and $\beta$ as above,
\begin{equation}
\label{eq:shrinkrate}
 \leb\bigl( \FM_q^{-n}([\alpha,\beta]) \bigr) \sim \frac{\log\left({\beta}/{\alpha}\right)}{\log(n)}  \quad\text{as $n\to\infty$}\,,
\end{equation}
which provides a precise shrinking rate of the Lebesgue volume along the sequence~$\bigl(\FM_q^{-n}([\alpha,\beta])\bigr)_n$. It shows that this rate is structurally the same for all such intervals~$[\alpha,\beta]$ with positive length (i.e., $\alpha\not=\beta$), but its scale depends logarithmically on~$\beta/\alpha$. For singletons, i.e., if $\alpha=\beta$, the expressions on both sides of~\eqref{eq:shrinkrate} are identically vanishing.

Our proof technique for Theorem~\ref{thm:dyn_intervals} does not apply to the case $q=3$ as the classical Farey map~$\FM_3$ is not an AFN-map. See Section~\ref{sec:AFN} for further details. For $q=3$, Theorem~\ref{thm:dyn_intervals} is established in~\cite{MR3052943} with a different proof, which however does not apply to $q>3$. Nevertheless, as Theorem~\ref{thm:dyn_intervals} is valid for all odd $q\geq 3$, we will use it in this generality to establish some of its consequences.

Theorem~\ref{thm:dyn_intervals} (for any odd $q\geq 3$) is crucial for the two results stated in Theorems~\ref{thm:jac_weight_lim_intro} and~\ref{thm:weighted_distr} below, both of which can be understood as a equidistribution result for \emph{weighted Dirac combs}.
For these results, we set
\begin{equation}\label{eq:Lambda}
\Lambda_q \coloneqq \left\{ g_{q,k}^{-1}, (Qg_{q,k})^{-1} : k \in \{(q+1)/2,\ldots, q-1\} \right\}\,,
\end{equation}
which collects the inverses of the elements acting in the branches of the generalized Farey map~$\FM_q$ (see~\eqref{eq:FM1}-\eqref{eq:FM2}).
The semigroup in~$\PGL_2(\R)$ generated by $\Lambda_q$ is free (as we will discuss in Section~\ref{sec:origin}),
and hence every element in this semigroup can be understood as a unique word over the alphabet~$\Lambda_q$.
For $n\in\N$, we let $W_{q,n}$ denote the set of words of length~$n$ over~$\Lambda_q$.

\begin{thm}\label{thm:jac_weight_lim_intro}
Let $q\geq 3$ odd. For each $x\in (0,1]$ we have
\[
 \starlim_{n\to\infty}\, x \log(n) \sum_{h\in W_{q,n}} |h'(x)|\delta_{h.x} = m\,.
\]
\end{thm}
We note that the statement of Theorem~\ref{thm:jac_weight_lim_intro} obviously does not extend to $x=0$. 

Recall that  $\cw_q \coloneqq 2\cos(\pi/q)$ is the cusp width of~$\Gamma_q$.
We call a pair $(a,c)\in \Z[\cw_q]\times\Z[\cw_q]$ a \emph{reduced fraction} if there exists an element in~$\Gamma_q$ of the form
\[
\bmat{a}{*}{c}{*}\,.
\]
Clearly, if $(a,c)$ is a reduced fraction, then also $(-a,-c)$ is a reduced fraction. 
For this reason, we identify $(a,c)$ and $(-a,-c)$ and, by slight abuse of notion, continue to denote the equivalence class $[(a,c)] = \{(a,c), (-a,-c)\}$ by $(a,c)$ and also continue to call it a reduced fraction. 
We further identify the (equivalence class of the) reduced fraction~$(a,c)$ with the value $a/c \in \R\cup\{\infty\}$. 
Clearly, each point in $\Gamma_q.\infty$ can be represented by a reduced fraction. 
The representing reduced fraction is indeed unique, which is another justification for descending from the original pairs in $\Z[\cw_q]^2$ to the equivalence classes (see Section~\ref{sec:proofweight} for details). 
For $q=3$, this notion of reduced fraction is the classical notion of reduced fractions from integer arithmetic.
For any $q\geq 3$, $n\in \N$  and  reduced fraction $v/w\in \Gamma_q.\infty \cap [0,1]$, we let
\[
 \RF_{q,n}(v,w) \coloneqq \left\{ \text{$(r,s)$ reduced fraction} :  \frac{r}{s}\in \FM_q^{-n}\left(\frac{v}{w}\right) \right\}
\]
denote the preimages of $v/w$ under $\FM_q^n$ in reduced fraction representation representing \emph{cusp points of level $n$} seen from $v/w$.
We emphasize that the level is indeed uniquely determined as soon as $v/w\not=0$.
A rather immediate consequence of Theorem~\ref{thm:jac_weight_lim_intro} is the following equidistribution result for weighted Dirac combs supported on cusp points of increasing level.

\begin{thm}\label{thm:weighted_distr}
Let $q\geq 3$ odd. For each reduced fraction $(v,w)\in\Gamma_q.\infty \cap (0,1]$ we have
\[
 \starlim_{n\to\infty} c_{v/w}\,vw\,\log(n) \sum_{ (r,s)\in \RF_{q,n}(v,w) } \frac1{s^2}\,\delta_{r/s} = m\,,
\]
where $c_1\coloneqq 2$ and $c_x\coloneqq 1$ for $x\not=1$.
\end{thm}

We emphasize that the values $vw$ and $s^2$ in Theorem~\ref{thm:weighted_distr} are independent of the choice of the representing pair in $\Z[\cw_q]^2$ for the respective reduced fraction and hence are indeed well defined.

As Theorem~\ref{thm:jac_weight_lim_intro}, Theorem~\ref{thm:weighted_distr} obviously does not extend to~$0$. Further,  \cite[Theorem 1.2]{MR3052943} is covered by  Theorem~\ref{thm:weighted_distr} as the special case with $q=3$.

\subsection{Key elements of proofs, and structure of article}

Central for the proofs of Theorems~\ref{thm:dyn_intervals}--\ref{thm:weighted_distr} are methods from infinite ergodic theory and transfer operator techniques. For Theorem~\ref{thm:dyn_intervals} we will take advantage of a result by Melbourne and Terhesiu~\cite{melbourne_operator_2012} on the asymptotic behavior of iterates of the transfer operator of a topologically mixing AFN-map. We restate this result in a specialized version as Theorem~\ref{thm:MT} in Section~\ref{sec:fundpropFarey}. The main bulk of work for establishing Theorem~\ref{thm:dyn_intervals} is therefore to show that the generalized Farey maps are eligible for this result and to determine all necessary entities in quite some detail. For that reason we prove---in Section~\ref{sec:fundpropFarey}---that the generalized Farey maps~$\FM_q$ for odd $q\geq 5$ are indeed topologically mixing and  AFN-maps. At this point we observe the subtlety that the Farey map~$\FM_3$ is not an AFN-map. We further provide the Perron--Frobenius operator, the invariant measure, and the transfer operator in explicit form. Moreover, we undertake a detailed study of the tail probabilities and their summation functions. Section~\ref{sec:proofthm11} is devoted to combining the results from Section~\ref{sec:fundpropFarey} with \cite{melbourne_operator_2012} to a proof of Theorem~\ref{thm:dyn_intervals}.

The proof of Theorem~\ref{thm:jac_weight_lim_intro} makes crucial use of Theorem~\ref{thm:dyn_intervals} in combination with a detailed comparison of the weighted sums of Dirac measures in the statement of Theorem~\ref{thm:jac_weight_lim_intro} and suitably restricted and rescaled Lebesgue measures. We provide the details in Section~\ref{sec:proofweight}. Theorem~\ref{thm:weighted_distr} is a rather straightforward consequence of Theorem~\ref{thm:jac_weight_lim_intro}; a proof is discussed in Section~\ref{sec:proofweight}.

\subsection{Related results}\label{sec:motivation}

As mentioned above, the generalized Farey maps as well as the Perron--Frobenius operators arose in~\cite{Pohl_Symdyn2d, Moeller_Pohl, Pohl_spectral_hecke}. See also Section~\ref{sec:origin}. We here restrict the considerations to Hecke triangle groups with \emph{odd} angle parameter. However, we expect that similar results are valid for the Hecke triangle groups with \emph{even} angle parameter, which are discussed alongside with the odd angle parameter groups in~\cite{Pohl_spectral_hecke}. In the even angle parameter situation, the generalized Farey ``map'' is not a genuine map but a relation with weights, and the associated Perron--Frobenius operator features additional weights. See~\cite{Pohl_spectral_hecke}. The investigations in~\cite{Moeller_Pohl, Pohl_spectral_hecke} show an intimate relation between eigenfunctions with eigenvalue~$1$ of the Perron--Frobenius operator and Maass cusp forms for the Hecke triangle groups, as well as between the Perron--Frobenius operator and the Selberg zeta function, thereby also providing a generalized Gauss map. See also~\cite{Pohl_mcf_general, Adam_Pohl, FP_szf}. It would be interesting to understand if these relations can be used to improve on our results, probably in regard to higher order asymptotics as in~\cite{MR3459025}. Moreover, it would be interesting to see if our results can be generalized to Hecke triangle groups of infinite co-area or other classes of Fuchsian groups. The necessary discretizations and Perron--Frobenius operators with similar properties are already available. See, e.g., \cite{Bruggeman_Pohl, Pohl_Symdyn2d, Pohl_Wabnitz, Wabnitz}.

In \cite{MR3052943} the equidistribution results for $q=3$ (i.e., Theorems~\ref{thm:dyn_intervals} and~\ref{thm:weighted_distr} for $\PSL_2(\Z)$) were compared with the famous equidistribution results of the Farey sequence according to Franel and Landau~\cite{zbMATH02595443}. It would be interesting to consider analogous results for Hecke triangle groups for the distribution of cusp points determined by the size of the denominators of their defining reduced fractions.
We emphasize that our results are closely related to the notion of measure-theoretical mixing for dynamical systems. Indeed, the following consequence of Theorem~\ref{thm:MT} below applied to our situation lies at the heart of the proof of our main theorem and reveals the following  mixing property: For all intervals $U,V\subset (0,1]$ that are bounded away from $0$,
\[
\log(n)\,\mu\bigl( \FM^{-n}(U)\cap V \bigr) \stackrel{n\to\infty}{\longrightarrow} \mu(U)\mu(V)\,.
\]
With Theorem~\ref{thm:dyn_intervals} we obtain
\[
\log(n)\,\leb\bigl( \FM^{-n}(U)\cap V \bigr) \stackrel{n\to\infty}{\longrightarrow} \mu(U)\leb(V)\,.
\]
For the first situation, Aaronson~\cite{MR0584018} introduced the concept of \emph{ratio mixing} for this type of mixing behavior in the context of infinite ergodic theory. See also~\cite{MR2303796, MR3345166, MR3867232}.
Moreover, the same discretizations giving rise to the generalized Farey maps also give rise to reduction theories for indefinite quadratic forms associated to Hecke triangle groups~\cite{Pohl_Spratte}. An investigation of potential relations between our equidistribution results and properties of these reduction theories are of interest.

\begin{ack}
The research of LB and AP is funded by the Deutsche Forschungsgemeinschaft (DFG, German Research Foundation) -- project no.~441868048 (Priority Program~2026 ``Geometry at Infinity'').
\end{ack}

\section{Fundamental properties of the generalized Farey maps}
\label{sec:fundpropFarey}

In this section we will establish several fundamental properties of the generalized Farey map and of related objects such as transfer operators and tail probabilities, which we will need for the proof of Theorem~\ref{thm:dyn_intervals}. A crucial role in our proof of Theorem~\ref{thm:dyn_intervals} will be played by the following result by Melbourne and Terhesiu, which we state here initially without further explanation and refer to the discussions further below for details and definitions. Throughout let $X\coloneqq [0,1]$ denote the closed unit interval and endow it with its Borel $\sigma$-algebra, which we denote~$\mc B = \mc B_X$.

\begin{thm}[Special case of {\cite[Theorem~1.1]{melbourne_operator_2012}}]
\label{thm:MT}
Let $H\colon X\to X$ be a topologically mixing AFN-map and let $X'$ denote the complement of the indifferent fixed points of~$H$ in~$X$.
Suppose that $\mu$ is a $\sigma$-finite $H$-invariant measure that is absolutely continuous to the Lebesgue measure on~$X$, and let $h$ be the density of~$\mu$. Further suppose that $(X, \mc B, H,\mu)$ has regularly varying tail probabilities with exponent~$1$. Let $\wh H$ denote the transfer operator associated to~$H$ and~$\mu$, and let $w$ denote the tail probabilities summation function.
Then
\[
 \lim_{n\to\infty} w(n)\wh H^n(v) =  \int_X v\,d\mu
\]
uniformly on compact subsets of~$X'$ for all~$v\colon X\to\R$ of the form~$v=u/h$, where $u$ is Riemann integrable on~$X$.
\end{thm}

In the following subsections we will show that the generalized Farey map~$\FM_q$ satisfies the hypotheses of Theorem~\ref{thm:MT} for odd $q\geq 5$.
In Section~\ref{sec:proofthm11} we will combine our results with Theorem~\ref{thm:MT} to a proof of Theorem~\ref{thm:dyn_intervals}. For $q=3$, the map~$\FM_q$ is not eligible for Theorem~\ref{thm:MT} because, as we will see below, $\FM_3$ is not an AFN-map. Nevertheless we will consider throughout all odd $q\geq 3$. If there is no need to distinguish between different values for~$q$ or if the range for $q$ is clear from the context, we will omit $q$ from the subscripts to simplify the notation.

\subsection{Origin of the generalized Farey maps}
\label{sec:origin}

The generalized Farey map~$\FM$ deduces from a symbolic dynamics for the geodesic flow on the Hecke triangle orbisurface~$\Gamma\backslash\h$, which is developed in~\cite{Pohl_Symdyn2d}. Alternatively, it can be developed from a symbolic dynamics for the billiard flow on the triangle surface~$\wt\Gamma\backslash\h$, where $\wt\Gamma\coloneqq \langle \Gamma, Q\rangle$ is the triangle group underlying~$\Gamma$, i.e., the subgroup of~$\PGL_2(\R)$ generated by~$\Gamma$ and~$Q$. See~\cite{Pohl_spectral_hecke}. We provide a few more details.

For any subset $U\subseteq \R$ we set
\[
U_\st \coloneqq U\setminus\Gamma.\infty\,.
\]
Here, the subscript $\st$ refers to ``strong,'' a notion which is explained in detail in~\cite{Pohl_Symdyn2d} but is not essential for our discussions. The construction  from~\cite{Pohl_Symdyn2d} (for $\Gamma\backslash\h$ see in particular the sequence of examples therein) provides a conjugation (more precisely, an almost-conjugation; see below) of the geodesic flow on~$\Gamma\backslash\h$ to the bijective map
\[
H\colon (0,\infty)_\st\times (-\infty,0)_\st \to (0,\infty)_\st\times (-\infty,0)_\st
\]
that is determined by
\begin{align*}
(g_k^{-1}.0, g_k^{-1}.\infty)_\st\times (-\infty,0)_\st & \to (0,\infty)_\st\times (g_k.(-\infty), g_k.0)_\st\,,
\\
(x,y) &\mapsto (g_k.x,g_k.y)
\end{align*}
for $k\in\{1,\ldots, q-1\}$. The relation between the geodesic flow on~$\Gamma\backslash\h$ and the map~$H$ is as follows: using the upper half plane model for~$\h$, for any geodesic $\wh\gamma$ on~$\Gamma\backslash\h$ we fix a representing geodesic~$\gamma$ on~$\h$ which satisfies $\gamma(+\infty)\in (0,\infty)$ and $\gamma(-\infty)\in (-\infty,0)$ (this is indeed possible). This is equivalent to asking that $\gamma$ intersects the imaginary axis~$i\R_{>0}$, say at time~$t_0$ (i.e., $\gamma(t_0)\in i\R_{>0}$), and travels from left to right (in the euclidean coordinates of~$\h$). Then $\wh\gamma$ intersects $\Gamma\backslash i\R_{>0}$ also at time~$t_0$. Suppose that $(t_j)_{j\in J}$ with $J\subseteq \Z$ is the sequence of intersection times of $\wh\gamma$ with $\Gamma\backslash i\R_{>0}$. Each intersection time is corresponding to a (unique) representing geodesic on~$\h$ intersecting $i\R_{>0}$ at that time. Let $(\gamma_j)_{j\in J}$ be the corresponding sequence of geodesics on~$\h$ with endpoints sequence $\bigl((\gamma_j(+\infty), \gamma_j(-\infty))\bigr)_{j\in J}$. Note that $\gamma=\gamma_0$. Applying $H$ iteratively to~$(\gamma(+\infty), \gamma(-\infty))$ moves along this sequence in positive index direction; applying $H^{-1}$ allows us to move in negative index direction. In other words, the application of~$H$ corresponds to moving along~$\wh\gamma$ from time $t_j$ to $t_{j+1}$. In the case that $\gamma(+\infty)\in\Gamma.\infty$ or $\gamma(-\infty)\in\Gamma.\infty$, the map~$H$ is not defined initially. In these cases, $J\not=\Z$. However, if $\gamma(+\infty)$ is not of the form $g_k^{-1}.0$ or $g_k^{-1}.\infty$ for some $k\in\{1,\ldots, q-1\}$, then we can continuously extend $H$ to $\gamma(+\infty)$. Analogously, for $H^{-1}$ and $\gamma(-\infty)$ of the form $g_k.(-\infty)$ or $g_k.0$. In the remaining cases, there are no future (or past) intersection times after $t_0$ (before $t_0$). For a detailed discussion we refer to~\cite{Pohl_Symdyn2d}.

For our purposes, we restrict to the ``forward direction" of~$H$, meaning to the map
\[
H_+ \colon (0,\infty)_\st \to (0,\infty)_\st
\]
that is determined by the bijections
\[
(g_k^{-1}.0,g_k^{-1}.\infty)_\st \to (0,\infty)_\st\,,\quad x\mapsto g_k.x\,,
\]
for $k\in\{1,\ldots, q-1\}$. This map enjoys a symmetry under the action of~$Q$, i.e., $Q.H_+(x) = H_+(Q.x)$. Factoring out this symmetry results in the map
\[
F_+ \colon (0,1)_\st \to (0,1)_\st
\]
that is given by the bijections
\begin{align*}
\bigl(g_k^{-1}.0,g_k^{-1}.1\bigr)_\st &\to (0,1)_\st\,, \quad x\mapsto g_k.x\,,
\intertext{and}
\bigl((Qg_k)^{-1}.1, (Qg_k)^{-1}.0\bigr)_\st & \to (0,1)_\st\,,\quad x\mapsto Qg_k.x\,,
\end{align*}
for $k\in\{(q+1)/2,\ldots, q-1\}$. The continuous extension of~$F_+$ to all of $X=[0,1]$ is the generalized Farey map~$\FM$.

The conjugation relation between~$H$ and the geodesic flow implies that the subsemigroup of~$\Gamma$ generated by $\{g_k^{-1}: k=1,\ldots, q-1\}$ is free. This property gets inherited under the $Q$-symmetry and shows that the semigroup in~$\PGL_2(\R)$ generated by~$\Lambda$ (see~\eqref{eq:Lambda}) is indeed free.

\subsection{Topologically mixing}
\label{sec:topmix}

With the following proposition we show that the Farey map~$\FM$ is topologically mixing.
Essentially, this is a consequence of the bijections in~\eqref{eq:FM1} and \eqref{eq:FM2} being sufficiently expanding and having full images.
We provide here another proof, taking advantage of the origin of~$\FM$ (see Section~\ref{sec:origin}) and results established in~\cite{Pohl_Symdyn2d}.

\begin{prop}\label{prop:Fareytopmix}
The Farey map~$\FM$ is topologically mixing.
\end{prop}

\begin{proof}
We set
\[
\eta\coloneqq \left\{ \left[g_k^{-1}.0,g_k^{-1}.1\right], \left[(Qg_k)^{-1}.1, (Qg_k)^{-1}.0\right] : k = \frac{q+1}2,\ldots, q-1  \right\}
\]
and consider, for $n\in\N_0$, the refinement
\begin{equation}
\label{eq:refinement}
\eta_0^n\coloneqq \bigvee_{j=0}^n \FM^{-j}(\eta) = \left\{  \bigcap_{j=0}^n \FM^{-j}(A_j) : A_j \in\eta,\ j\in\{0,\ldots, n\}\right\}\,.
\end{equation}
This family consists of intervals which are iterative refinements of the initial intervals in~$\eta$. For any choice of $A_j\in\eta$, $j\in\{0,\ldots, n\}$, the evaluation of~\eqref{eq:FM1} and~\eqref{eq:FM2} shows that
\begin{equation}
\label{eq:bijFMn}
\bigcap_{j=0}^n \FM^{-j}(A_j) \to [0,1]\,,\quad x\mapsto \FM^{n+1}(x)\,,
\end{equation}
is a bijection. Further, from the origin of~$\FM$ and the results proved in~\cite[Proposition~8.28, Remark~8.32]{Pohl_Symdyn2d} (initially for the map~$H_+$ from Section~\ref{sec:origin} and then inherited to~$\FM$ by $Q$-symmetry) we know that refinement intervals in the family~\eqref{eq:refinement} become arbitrarily small as $n\to\infty$.

Let $B_1, B_2$ be non-empty, open subsets of~$X$. Then we find $N\in\N_0$ such that $B_1$ contains some interval from~$\eta_0^N$. Hence, using the bijectivity of~\eqref{eq:bijFMn}, we obtain that $\FM^{n+1}(B_1) \cap B_2 \not=\emptyset$ for all $n\geq N$, showing that $\FM$ is topologically mixing.
\end{proof}

\subsection{AFN-maps}
\label{sec:AFN}

In this section we show that the generalized Farey map~$\FM=\FM_q$ for odd $q\geq 5$ is an AFN-map, but not for $q=3$. We further discuss that for any odd $q\geq 3$, the map~$\FM$ is ergodic and conservative for any measure on~$X$ that is absolutely continuous to the Lebesgue measure~$\leb$. We start by recalling the definition of AFN-maps and the auxiliary notion of piecewise monotonic systems, which both were introduced in~\cite{zweimuller_structure_2000}.

\begin{defi}
	A \emph{piecewise monotonic system} is a triple $(X, T, \prt)$ where $X$ is the union of some finite family of disjoint bounded open intervals, $\prt$ is a collection of non-empty, pairwise disjoint open subintervals with the following two properties:
	With
	\begin{equation*}
		X_{\prt} \coloneqq \bigcup_{\inter\in\prt}\inter
	\end{equation*}
	we have
	\begin{equation*}
		\leb\left(X\setminus X_{\prt}\right)=0\,.
	\end{equation*}
	Further, $T\colon X\to X$ is a map such that $T\lvert_{\inter}$ is continuous and strictly monotonic for each $\inter\in\prt$.
\end{defi}

\begin{defi}\label{def:afn}
	Let $(X, T, \prt)$ be a piecewise monotonic system.
	Suppose that $T\lvert_\inter$ is twice differentiable for each $\inter\in\prt$.
	Then $T$ is called an \emph{AFN-map} if it satisfies the following properties:
	\begin{enumerate}
		\item[(A)] \emph{Adler's condition: } The map~$T''/(T')^2$ is bounded on $X_{\prt}$.
		\item[(F)] \emph{Finite image condition: } The set~$T(\prt)=\{T(\inter):\inter\in\prt\}$ is finite.
		\item[(N)] \emph{Non-uniformly expanding: } There is a finite set $\prt_0\subseteq\prt$ such that each $\inter\in\prt_0$ has a (unique) element~$x_{\inter}$ in the boundary of~$I$ with the following properties:
		\begin{itemize}
			\item We have
			\begin{align*}
				\lim\limits_{\substack{ x\in\inter\\ x\to x_{\inter}}} T(x) = x_\inter
				\shortintertext{and}
				\lim\limits_{\substack{ x\in\inter\\ x\to x_{\inter}}} T'(x) = 1\,.
			\end{align*}
			\item The map~$T'$ decreases on~$(-\infty, x_{\inter})\cap\inter$ or on~$(x_{\inter},\infty)\cap\inter$, depending on which of these sets is non-empty.
            \item The map~$T$ is uniformly expanding on sets bounded away from $\{x_{\inter}:\inter\in\prt_0\}$.
		That is, for all $\varepsilon>0$ there exists~$\varrho >1$ such that for all $x\in X\setminus\bigcup_{\inter\in\prt_0}((x_{\inter}-\varepsilon,x_{\inter}+\varepsilon)\cap\inter)$ we have
		\begin{equation*}
			|T'(x)|\geq \varrho\,.
		\end{equation*}
		\end{itemize}
    \end{enumerate}
    The set~$\prt_0$ is indeed unique. The elements~$x_\inter$, $\inter\in\prt_0$, are called the \emph{indifferent fixed points} of~$T$. The complement of the set of indifferent fixed points of~$T$ is denoted~$X'$.
\end{defi}

In regard to the generalized Farey map~$\FM$, we set throughout
\[
 \prt \coloneqq \left\{ g.(0,1): g\in \Lambda \right\}\,,
\]
where $\Lambda$ is the set defined in~\eqref{eq:Lambda}.
With respect to the Lebesgue measure~$\leb$ and hence any measure absolutely continuous to~$\leb$, the system~$\prt$ is a measure-theoretic partition of~$X=[0,1]$. One easily checks that $0$ is the only indifferent fixed point of~$\FM$. The proof of Proposition~\ref{prop:FareyAFN} below shows that for~$\FM=\FM_q$ with $q\geq 5$ odd, we have
\[
\prt_0 = \left\{ \bigl(g_{q-1}^{-1}.0,g_{q-1}^{-1}.1\bigr) \right\} = \left\{ \left( 0, \frac{1}{\cw+1}\right) \right\}
\]
and
\begin{equation}
\label{eq:complindfixed}
X' = (0,1]\,.
\end{equation}

\begin{prop}\label{prop:FareyAFN}
For every odd $q\geq 5$, the Farey map~$\FM=\FM_q$ is an AFN-map.
\end{prop}

\begin{proof}
	The union of all elements in the collection $\prt$ clearly has full Lebesgue measure in~$X$.
	On each interval in $\prt$, the Farey map $\FM$ corresponds to a fractional linear transformation associated to some $g=\textbmat{a}{b}{c}{d}\in \PSL(2,\R)$.
	As the derivative of the action of $g$ is given by $g'(x) = \det(g)(cx+d)^{-2}$ for all $x\in\R$ such that $g.x\neq \infty$, it follows that $(X,\FM,\prt)$ is piecewise monotonic.
	Since $\prt$ is a finite set, it is obvious that $\FM$ satisfies the finite image condition.
	In particular, each branch (i.e., $\FM\vert_\inter$ for $\inter\in\prt$) is surjective onto $(0,1)$.
	
	The second derivative of a fractional linear transformation $g=\textbmat{a}{b}{c}{d}\in \PSL(2,\R)$ is given by $g''(x)=-2c\det(g)(cx+d)^{-3}$.
	It follows that
	\begin{equation*}
		\frac{g''(x)}{(g'(x))^2} = -2c\det(g)(cx+d)\,.
	\end{equation*}
	Since $\FM$ consists of finitely many branches, there is a maximum value which bounds this quotient uniformly, yielding the Adler property for $\FM$.
	
	It remains to show that $\FM$ is non-uniformly expanding.
	We find that 
	\begin{equation*}
		\prt_0 = \left\{ \bigl(g_{q-1}^{-1}.0,g_{q-1}^{-1}.1\bigr) \right\}
	\end{equation*}
	as the point $0$ in the boundary of the interval in $\prt_0$ satisfies
	\begin{align*}
		\lim_{x\to 0}\FM(x) &= \lim_{x\to 0}g_{q-1}.x = \lim_{x\to 0}\frac{x}{-\cw x+1} = 0
		\shortintertext{as well as}
		\lim_{x\to 0}\FM'(x) &= \lim_{x\to 0}g_{q-1}'(x) = \lim_{x\to 0}\frac{1}{(-\cw x+1)^2} = 1\,.
	\end{align*}
	In particular, for each $\varepsilon>0$ and $x\in g_{q-1}^{-1}.(0,1)\setminus(0,\varepsilon)$ we have
	\begin{equation}
	\label{eq:estFM1}
		\lvert \FM'(x)\rvert = g_{q-1}'(x) \geq  \frac{1}{(-\cw \varepsilon+1)^2} >1\,.
	\end{equation}
	For the last inequality we note that for $\eps > 1/(\cw +1)$, the set $g_{q-1}^{-1}.(0,1)\setminus(0,\varepsilon)$ is empty.
	Consider the matrix representation of $g_k$ given in~\eqref{eq:gk}. Then we have $cx+d>0$ for any $g=\textbmat{a}{b}{c}{d}\in \Lambda$ and $x\in g.(0,1)$.
	Let $k\in\{(q+1)/2,\dotsc,q-2\}$ and $x\in g_k^{-1}.(0,1)$.
	Then $\FM(x)=g_k.x$ and the coefficient~$c$ is negative, which implies
	\begin{equation*}
		g_k''(x)=-2c(cx+d)^{-3}>0\,.
	\end{equation*}
	Since $g_k^{-1}$ is monotone increasing, it follows that
	\begin{equation}
	\label{eq:estFM2}
		\lvert \FM'(x)\rvert = g_k'(x)\geq g_k'(g_k^{-1}.0) = \left(\frac{\sin(k\pi/q)}{\sin(\pi/q)}\right)^2
		\geq \cw^2\,.
	\end{equation}

	Let $x\in g_k^{-1}Q.(0,1)$ for $k\in\{(q+1)/2,\dotsc, q-1\}$.
	The coefficient $c$ of $Qg_k$ is positive, which implies $\frac{d}{dx}\lvert(Qg_k)'(x)\rvert = -2c(cx+d)^{-3}<0$.
	Consequently,
	\begin{equation}
	\label{eq:estFM3}
		\lvert \FM'(x)\rvert = \lvert(Qg_k)'(x)\rvert \geq \lvert(Qg_k)'(g_k^{-1}Q.0)\rvert = \left(\frac{\sin((k-1)\pi/q)}{\sin(\pi/q)}\right)^2\geq \cw^2\,.
	\end{equation}
	For $q\geq 5$, we have $\cw>1$. Thus, the combination of~\eqref{eq:estFM1}--\eqref{eq:estFM3} yields that $\FM$ is non-uniformly expanding. This completes the proof.
\end{proof}

We saw in the proof above that $|\FM_q'|\geq \cw_q$ on the partition elements in $\prt\setminus\prt_0$. For $q=3$ we have $\cw_3=1$, and indeed $|\FM_3'(1)|=1$ which means that $\FM_3$ does not satisfy Condition~(N) of an AFN-map.

We now turn to establish that the generalized Farey maps are ergodic and conservative with respect to any measure on~$X$ that is absolutely continuous to the Lebesgue measure~$\leb$. For the classical Farey map, i.e., for $\FM_3$, this is well-known. See~\cite{MR3585883}. 
In order to treat~$\FM_q$ for odd $q>3$, we first prove a general result on ergodicity and conservativity for topologically mixing AFN-maps. For any measure~$\nu$ on~$(X,\mc B)$ and any measurable set $B_1,B_2\in\mc B$ we write $B_1\stackrel{\nu}{=}B_2$ if $B_1$ and $B_2$ differ only by a $\nu$-null set, i.e., if $B_1\triangle B_2 = (B_1\setminus B_2) \cup (B_2\setminus B_1)$ is a $\nu$-null set.

\begin{lemma}
\label{lem:AFNerg}
	Any topologically transitive AFN-map on~$X$ is ergodic and conservative with respect to the Lebesgue measure~$\leb$ and hence any measure absolutely continuous to~$\leb$.
\end{lemma}
\begin{proof}
	Let $T$ be a topologically transitive AFN-map on $X=[0,1]$.
	It is proven in \cite[Theorem~1]{zweimuller_structure_2000} that there exists a finite number of pairwise disjoint open sets $X_1,\dotsc,X_m$ such that
	\begin{itemize}
		\item $T(X_i)\stackrel{\leb}{=}X_i$ and $T\lvert_{X_i}$ is ergodic and conservative with respect to~$\leb$.
		\item Each $X_i$ admits a finite partition into sets that are finite unions of open intervals.
	\end{itemize}
	We will now show that $X\stackrel{\leb}{=} X_1$, which immediately yields that $T$ is ergodic and conservative with respect to~$\leb$. To that end we first note that the property~(N) of non-uniform expansiveness of~$T$ from Definition~\ref{def:afn} implies that for each open, non-empty interval $I\subseteq X$ and each $n\in\N$ there exist an open interval $J\subseteq I$ and some $\varrho>1$ such that
	\[
	\left| \bigl(T^n\bigr)'\vert_J\right| \geq \varrho\,.
	\]
	In particular, $\leb(T^n(I))>0$. We now set $E\coloneqq X\setminus X_1$ and suppose $\leb(E)>0$, with the goal to establish a contradiction. As $E$ is a finite union of closed intervals (among which might be singletons), there exists an open, non-empty interval $B\subseteq E$. Since $T$ is topologically transitive, we find $n\in\N$ such that
	\[
	T^n (X_1) \cap B \not=\emptyset\,.
	\]
As $X_1$ decomposes into a finite union of open, non-empty intervals and $T$ is piecewise continuous, the stability of connectedness under continuous maps in combination with $B$ being an open (connected) interval shows that there is an open interval $I\subseteq X_1$ such that $T^n(I) \subseteq B$. By our previous argumentation, $\leb(T^n(I)) > 0$. This contradicts to $T(X_1) \stackrel{\leb}{=} X_1$. In turn, $\leb(E) = 0$.
\end{proof}

As topologically mixing maps are \emph{a fortiori} topologically transitive, the combination of Proposition~\ref{prop:FareyAFN} with Lemma~\ref{lem:AFNerg} yields the following corollary for odd $q\geq 5$. For $q=3$, it is in~\cite{MR3585883}.

\begin{cor}
\label{cor:ergcons}
	For all odd $q\geq 3$, the Farey map~$\FM=\FM_q$ is ergodic and conservative with respect to any measure on~$X$ that is absolutely continuous to the Lebesgue measure.
\end{cor}

\subsection{Transfer operators and invariant measure}
\label{sec:measure}

In this section, we establish the existence of an essentially unique $\sigma$-finite $\FM$-invariant measure on~$(X,\mc B)$ that is absolutely continuous to the Lebesgue measure~$\leb$ and we provide its density in explicit form. We further provide an explicit expression for the transfer operator associated to~$\FM$ and this measure. We emphasize that throughout this section we consider all odd $q\geq 3$.

For any measure~$\nu$ on~$(X,\mc B)$, the transfer operator
\[
L \colon L^1(X,\nu) \to L^1(X,\nu)
\]
associated to~$(\FM,\nu)$ is determined by the identity
\begin{equation}\label{eq:defTO}
\forall\, B\in\mc B\ \forall\, f\in L^1(X,\nu)\colon  \int_B L(f)\dd\nu = \int_{\FM^{-1}(B)} f\dd\nu\,.
\end{equation}
For $\nu$ being the Lebesgue measure~$m$, we call the transfer operator associated to~$(\FM,\leb)$ \emph{Perron--Frobenius operator} and denote it by~$\PF$. It has the explicit expression
\[
 \PF = \sum_{k=\frac{q+1}{2}}^{q-1} \tau(g_k) + \tau(Qg_k)\,,
\]
where
\[
\tau(g)f(x) \coloneqq \bigl|(g^{-1})'(x)\bigr|f\bigl( g^{-1}.x\bigr)
\]
for $g\in\PGL_2(\R)$ and $f\colon \R\setminus\{g.\infty\} \to \C$. See also~\cite{Pohl_spectral_hecke}, where it arises as the transfer operator with parameter~$1$ for the triangle group~$\langle\Gamma, Q\rangle$. The essentially unique eigenfunction of the Perron--Frobenius operator~$\PF$ with eigenvalue~$1$ is the density of an $\FM$-invariant $\sigma$-finite measure. In Proposition~\ref{prop:eigenfunctionh} we provide this eigenfunction, in Proposition~\ref{prop:invmeasure} we show the invariance of the associated measure (among other results).

\begin{prop}
\label{prop:eigenfunctionh}
	The map
	\begin{equation*}
		h\colon (0,1]\to\C,\quad x\mapsto\frac{1}{x}
	\end{equation*}
	is a real analytic eigenfunction of $\PF$ with eigenvalue $1$.
\end{prop}

\begin{proof}
We define the two operators
\[
\TO_0, \TO_1 \colon \Fct( \R_{>0}, \C)\to \Fct( \R_{>0}, \C)
\]
by
\begin{align*}
\TO_0f(x) & \coloneqq \sum_{k=1}^{q-1} f\bigl( g_k^{-1}.x \bigr)
\intertext{and}
\TO_1f(x) & \coloneqq \sum_{k=1}^{q-1} \tau(g_k)f(x) =  \sum_{k=1}^{q-1}  \bigl(g_k^{-1}\bigr)'(x)  f\bigl( g_k^{-1}.x \bigr)\,.
\end{align*}
We set $f,g\colon\R_{>0}\to\R$, $f(x) \coloneqq \log x$ and $g(x)\coloneqq 1/x = f'(x)$. We recall from~\eqref{eq:gk} that, for $k\in\Z$,
\[
 g_k^{-1} = \bmat{ s(k) }{ s(k+1) }{ s(k-1) }{ s(k) }\,.
\]
Then
{\allowdisplaybreaks
\begin{align*}
 \TO_0f(x) & = \sum_{k=1}^{q-1} \log\left( \frac{  x s(k) + s(k+1) }{ x s(k-1) +  s(k) } \right)
 \\
 & = \sum_{k=1}^{q-1} \left[ \log\bigl( x s(k) + s(k+1)\bigr)  - \log\bigl( x s(k-1) +  s(k) \bigr)\right]
 \\
 & = \log\bigl( x s(q-1) + s(q) \bigr)  - \log\bigl( x s(0) + s(1) \bigr)
 \\
 & = \log\bigl( x s(1) \bigr) - \log\bigl( s(1) \bigr)
 \\
 & = \log x = f(x)\,.
\end{align*}
}
Thus, $f$ is an eigenfunction of~$\TO_0$ with eigenvalue~$1$. Further, for any $x\in\R_{>0}$,
\begin{align*}
 g(x) & = \frac1x = \frac{d}{dx}f(x) = \frac{d}{dx}\TO_0f(x)
   = \sum_{k=1}^{q-1} (g_k^{-1})'(x)\, g\bigl(g_k^{-1}.x\bigr)
   = \TO_1g(x)\,.
\end{align*}
Thus, $g$ is an eigenfunction with eigenvalue~$1$ of~$\TO_1$.
We now combine this property with the facts that $g$ is $\tau(Q)$-invariant (i.e., $\tau(Q)g=g$), that $\tau$ is an equivariance when restricted to the group elements and function spaces considered here, and that $Qg_kQ = g_{q-k}$ for all~$k\in\Z$. We obtain
\begin{align*}
\PF g & = \sum_{k=\frac{q+1}{2}}^{q-1} \tau(g_k) g + \tau(Qg_k)g
\\
& = \sum_{k=\frac{q+1}{2}}^{q-1}\tau(g_k)g + \sum_{\ell=\frac{q+1}{2}}^{q-1} \tau(Qg_\ell Q)g
\\
& = \sum_{k=\frac{q+1}{2}}^{q-1}\tau(g_k)g + \sum_{\ell=1}^{\frac{q-1}{2}} \tau(g_\ell)g
\\
& = \TO_1g = g\,.
\end{align*}
Noting that $h=g\vert_{(0,1]}$ completes the proof.
\end{proof}

We let $\mu$ denote the measure on~$(X,\mc B)$ which is absolutely continuous to the Lebesgue measure~$m$ with density being the $\PF$-eigenfunction~$h\colon x\mapsto 1/x$ from Proposition~\ref{prop:eigenfunctionh}, i.e.,
\begin{equation}
\label{eq:defmu}
\dd\mu = h\dd\leb\,.
\end{equation}
We further let
\[
\wh\FM \colon L^1(X,\mu) \to L^1(X,\mu)
\]
denote the transfer operator associated to~$(\FM,\mu)$. The following proposition discusses the relation between the Perron--Frobenius operator~$\PF$ and the transfer operator~$\wh\FM$, thereby providing a convenient explicit formula for~$\wh\FM$. Moreover, it shows that $\mu$ is indeed $\FM$-invariant.

\begin{prop}
\label{prop:invmeasure}
	The measure~$\mu$ is infinite, $\sigma$-finite and $\FM_q$-invariant.
	The transfer operator $\wh{\FM}\colon L^1(\mu)\to L^1(\mu)$ of $\FM$ with respect to $\mu$ satisfies
	\begin{equation}\label{eq:perron-transfer}
		\wh{\FM}(f) = \PF\left(f h\right)h^{-1}\,.
	\end{equation}
\end{prop}

\begin{proof}
	It is clear that $\mu$ is $\sigma$-finite and infinite.
	Recall that $W_1=\Lambda$ is the alphabet generated by the inverse branches of $\FM$.
	The pre-image of a measurable set $B\subseteq \masterint$ is given by the $\mu$-almost disjoint union of $g.B$ for all $g\in W_1$.
	We have for all $f\in L^1(\mu)$,
	\begin{equation*}
		\int_{\FM^{-1}(B)} f\dd\mu
		= \sum_{g\in W_1}\int_{g.B} fh \dd \leb
		= \sum_{g\in W_1}\int_B (fh)(g.x) \lvert g'(x)\rvert \dd \leb(x)\,.
	\end{equation*}
	By the definition of $\PF$ this expression is equal to
	\begin{equation*}
		\int_B \PF(fh) \dd\leb = \int_B \PF(fh) h^{-1} \dd\mu\,.
	\end{equation*}
	It follows that $\wh{\FM}(f)=\PF(fh)h^{-1}$.
	Finally, as $h$ is an eigenfunction of eigenvalue $1$ of $\PF$, it holds that
	\begin{equation*}
		\wh{\FM}(\one) = h^{-1}\PF(h) = \one\,,
	\end{equation*}
	which implies that $\mu$ is $\FM$-invariant.
\end{proof}

\subsection{Tail probabilities}
\label{sec:tailprob}

For any subset $Y\subseteq X$ we let $\varphi_Y$ denote the \emph{first return time map} of~$Y$, that is,
\begin{equation}
\label{eq:firstreturn}
 \varphi_Y\colon Y\to \N\cup\{\infty\}\,,\quad \varphi_Y(x) \coloneqq \inf\{ n\in\N : \FM^n(x) \in Y\}\,,
\end{equation}
with the convention that $\inf\emptyset = \infty$.
The \emph{tail probabilities of~$(X,\mc B, \FM, \mu)$ associated to~$Y$} are the values $\mu(\varphi_Y > n)$
for~$n\in\N$ or, more precisely, the sequence of these values. This sequence is said to be \emph{regularly varying with exponent~$1$} if
\[
\mu(\varphi_Y > n) = \ell_Y(n) n^{-1}
\]
for a positive sequence~$(\ell_Y(n))_n$ satisfying
\[
\frac{\ell_Y( \lfloor \alpha  n \rfloor)}{\ell_Y(n)} \stackrel{n\to\infty}{\longrightarrow} 1
\]
for all~$\alpha > 0$. (In this case, $(\ell_Y(n))_n$ is called slowly varying.) The associated \emph{tail probabilities summation function} is
\[
w_Y(n)\coloneqq \sum_{j=1}^n \mu(\varphi_Y > j) \qquad\text{for $n\in\N$}.
\]
In this section we will show a rather uniform result for the tail probabilities of~$(X,\mc B, \FM, \mu)$, namely that for each compact subset~$C$ of~$X'=(0,1]$ we find a measurable set $Y(C)\subseteq (0,1]$ such that the tail probabilities $\bigl(\mu(\varphi_{Y(C)}>n)\bigr)_n$ are regularly varying with exponent~$1$ and the tail probabilities summation function~$w_{Y(C)}$ is asymptotic to~$\log n$ as $n\to\infty$. We emphasize that the asymptotics of the tail probabilities summation functions will indeed be found to be independent of~$C$, which is the uniformity needed for Theorem~\ref{thm:MT}. We refer to \cite{bingham_regular_1987} for further information on the concept of regularly varying, which is much more general than needed here.

Our proof for providing the sets~$Y(C)$ in Proposition~\ref{prop:tailprob} below is constructive, following an approach by \cite{zweimuller_ergodic_2000}. 
A crucial role will be played by the set
\begin{equation}\label{eq:defA}
 A \coloneqq (g_{q-1}^{-1}.1,1] = \left( \frac{1}{\cw+1}, 1 \right]\,,
\end{equation}
which has nicely structured propagation properties under~$\FM$. See Lemma~\ref{lem:Asweepout}. In particular, $A$ is a \emph{sweep-out set} for~$(\FM,\mu)$, i.e.,
\begin{equation}\label{eq:sweep}
\bigcup_{n\in\N_0} \FM^{-n}(A) \stackrel{\mu}{=} X\,,
\end{equation}
meaning that the sets on both sides equal up to a $\mu$-null set. As $(X,\mc B,\FM, \mu)$ is conservative and ergodic (see Corollary~\ref{cor:ergcons}), every set of positive and finite measure, such as~$A$, is a sweep-out set by~\cite[Lemma~2.4.3]{MR3585883}. 
However, for the set~$A$ a hands-on proof of the sweep-out property is straightforward and explicitly provides the $\mu$-null set in the equality~\eqref{eq:sweep}, for which reason we include it into the proof of Lemma~\ref{lem:Asweepout}.

\begin{lemma}\label{lem:Asweepout}
For all $n\in\N$ we have
\begin{equation}\label{eq:bulkinsweepout}
 \bigcup_{k=0}^{n-1} \FM^{-k}(A) = (g_{q-1}^{-n}.1,1] = \left( \frac{1}{n\cw+1}, 1\right]\,.
\end{equation}
In particular,
\begin{equation}\label{eq:infbulk_essall}
 \bigcup_{k=0}^{\infty} \FM^{-k}(A) = (0,1]\,,
\end{equation}
and the set~$A=(g_{q-1}^{-1}.1,1]$ is a sweep-out set for~$(\FM,\mu)$.
\end{lemma}

\begin{proof}
We start by showing the first statement by induction.
For $n=1$, \eqref{eq:bulkinsweepout} is immediate by the choice of~$A$.
For $n>1$ we first note that for each $B\subseteq (0,1]$ we have
\[
 \FM^{-1}(B) = \bigcup_{k=(q+1)/2}^{q-1} g_k^{-1}.B \cup (Qg_k)^{-1}.B\,,
\]
where, because of $A=(g_{q-1}^{-1}.1,1]$,
\[
 g_k^{-1}.B \subseteq A \qquad\text{for $k=\frac{q+1}{2},\ldots, q-2$}
\]
as well as
\[
 (Qg_k)^{-1}.B \subseteq A \qquad\text{for $k=\frac{q+1}{2},\ldots, q-1$}\,.
\]
Thus, if $B\subseteq A$ then
\begin{equation}\label{eq:satpreim}
 \FM^{-1}(B) \cup A = g_{q-1}^{-1}.B \cup A\,.
\end{equation}
We suppose now that \eqref{eq:bulkinsweepout} is valid for some~$n\in\N$.
Then, taking advantage of~\eqref{eq:satpreim}, we obtain
\begin{align*}
 \bigcup_{k=0}^{n} \FM^{-k}(A) & = \FM^{-1}\left( \bigcup_{k=0}^{n-1}\FM^{-k}(A) \right) \cup A
 \\
 & = \FM^{-1}\bigl( (g_{q-1}^{-n}.1,1] \bigr) \cup A
 \\
 & = g_{q-1}^{-1}.\bigl( (g_{q-1}^{-n}.1,1] \bigr) \cup (g_{q-1}^{-1}.1,1]
 \\
 & = (g_{q-1}^{-(n+1)}.1,1]\,,
\end{align*}
which establishes \eqref{eq:bulkinsweepout} for~$n+1$. Hence \eqref{eq:bulkinsweepout} is indeed valid for all~$n\in\N$.
For the other two statements of the lemma we use that
\[
 g_{q-1}^{-n}.1 = \frac{1}{n\cw + 1} \qquad\text{for $n\in\N$}
\]
and hence
\[
 \bigcup_{k=0}^\infty \FM^{-k}(A) = (0,1]\,.
\]
Because $\{0\}$ is a $\mu$-null set, it follows immediately that $A$ is a sweep-out set for $(\FM,\mu)$.
\end{proof}

For any subset $Y\subseteq X$, we let $\tilde\varphi_Y$ denote the \emph{first hitting time map} of~$Y$, that is,
\begin{equation}
\label{eq:firsthit}
 \tilde\varphi_Y\colon X\to\N\cup\{\infty\}\,,\quad \tilde\varphi_Y(x) \coloneqq \inf\{ n\in\N : \FM^n(x) \in Y\}\,,
\end{equation}
with the convention that $\inf\emptyset = \infty$. We emphasize that the domains of $\varphi_Y$ (see~\eqref{eq:firstreturn}) and $\tilde\varphi_Y$ differ. Restricted to~$Y$, both maps are identical. 
The sweep-out property for $(\FM,\mu)$ yields that the first hitting time of~$Y$ is almost surely finite if~$Y$ has positive and finite measure.
The first hitting time map will simplify the proof of Proposition~\ref{prop:tailprob} below. We emphasize that the proof of this proposition is indeed constructive.

For any positive sequences $(a_n)_n, (b_n)_n$ we write $a_n\sim b_n$ if they are asymptotic for $n\to\infty$, i.e., if
\[
\lim_{n\to\infty} \frac{a_n}{b_n} = 1\,.
\]

\begin{prop}
\label{prop:tailprob}
The first return time map $\varphi_A\colon A\to \N\cup\{\infty\}$ satisfies
\[
 \mu(\varphi_A>n) \sim n^{-1}\qquad \text{as $n\to\infty$}\,.
\]
Moreover,  for every compact set $C\subseteq (0,1]$, there exists a measurable set $Y(C)\subseteq (0,1]$ that contains $C$ and whose first return time map~$\varphi_{Y(C)}$ satisfies
\[
 \mu(\varphi_{Y(C)}>n) \sim n^{-1} \qquad \text{as $n\to\infty$}\,.
\]
\end{prop}

\begin{proof}
In what follows we will provide a construction of the measurable sets~$Y(C)$ in the second statement. These sets will be using the set~$A$ in such a way that the proof will show that the first statement follows immediately from a special case in the proof of the second statement. For this reason, we will discuss the second statement only.

Let $C$ be a compact subset of~$(0,1]$. By Lemma~\ref{lem:Asweepout} we find $N\in\N_0$ such that
\begin{equation}\label{eq:deftildeA}
 C \subseteq \bigcup_{k=0}^N \FM^{-k}(A) \eqqcolon Y(C)\,.
\end{equation}
To simplify notation, let $Y\coloneqq Y(C)$.
We reformulate the mass of $\{\varphi_{Y}>n\}$ by using the $\FM$-invariance of~$\mu$ and the dis\-joint\-ness of level sets of the first return time map for different times as follows:
\begin{align*}
\mu\left( Y^\complement \cap  \bigl\{ \wt\varphi_{Y} = n\bigr\} \right)
& = \mu\left( \FM^{-1} \left( Y^\complement \cap \bigl\{ \wt\varphi_{Y} = n\bigr\} \right)  \right)
\\
& = \mu\bigl( \varphi_{Y} = n+1 \bigr) + \mu\left( Y^\complement \cap \bigl\{ \wt\varphi_{Y} = n+1 \bigr\} \right)
\\
& = \mu\bigl( \varphi_{Y} = n+1 \bigr) + \mu\bigl( \varphi_{Y} = n+2 \bigr)
\\
& \hphantom{\mu\bigl( \varphi_{Y} = n+1 \bigr)} + \mu\left( Y^\complement \cap \bigl\{ \wt\varphi_{Y} = n+2 \bigr\} \right)
\\
& = \sum_{k=1}^\infty \mu\bigl( \varphi_{Y} = n+k \bigr) = \mu\bigl( \varphi_{Y} > n\bigr)\,,
\end{align*}
where we used the sweep-out property of~$A$ and induction in the penultimate equality.
Further we have, using Lemma~\ref{lem:Asweepout} for the last equality,
\begin{align*}
 Y^\complement \cap  \bigl\{ \wt\varphi_{Y} = n\bigr\} & = \FM^{-n}(Y) \setminus \bigcup_{j=0}^{n-1} \FM^{-j}(Y)
 \\
 & = \left( \bigcup_{j=0}^n \FM^{-n}(Y) \right) \setminus \left(\bigcup_{j=0}^{n-1} \FM^{-j}(Y) \right)
 \\
 & = \left( \bigcup_{j=0}^n \bigcup_{k=0}^N \FM^{-(j+k)}(A) \right) \setminus \left( \bigcup_{j=0}^{n-1} \bigcup_{k=0}^N \FM^{-(j+k)}(A) \right)
 \\
 & = \bigcup_{\ell=0}^{N+n} \FM^{-\ell}(A) \setminus \bigcup_{m=0}^{N+n-1} \FM^{-m}(A)
 \\
 & = \bigl( g_{q-1}^{-(N+n+1)}.1, g_{q-1}^{-(N+n)}.1 \bigr]\,.
\end{align*}
Combining this set equality with the reformulation of the mass above, and recalling the definition of $\mu$ in~\eqref{eq:defmu}, we obtain
\[
 \mu\bigl(\varphi_{Y} > n\bigr) = \mu\bigl(\bigl( g_{q-1}^{-(N+n+1)}.1, g_{q-1}^{-(N+n)}.1 \bigr]\bigr) = \log\left( 1 + \frac{\cw}{(N+n)\cw + 1} \right)\,.
\]
Using the bounds $x/(1+x) \leq \log(1+x) \leq x$, valid for $x>-1$, we conclude that
\[
 \frac{\cw}{(N+n+1)\cw + 1} \leq \mu\bigl( \varphi_{Y} > n \bigr) \leq \frac{\cw}{(N+n)\cw +1}\,.
\]
Thus,
\[
 \mu\bigl( \varphi_{Y} > n \bigr) \sim n^{-1} \qquad \text{as $n\to\infty$},
\]
which is the second statement of the proposition.
Considering $N=0$ in~\eqref{eq:deftildeA} (and picking, e.\,g., $C=\emptyset$), we obtain the first statement as well.
\end{proof}

In the following proposition we provide asymptotics for the summation functions of the tail probabilities in Proposition~\ref{prop:tailprob}. As all of these tail probabilities enjoy the same asymptotics for $n\to\infty$, i.e., independent of the specific set~$Y(C)$, we omit the subscript $Y(C)$.

\begin{prop}
\label{prop:tailsum}
For the tail probabilities summation functions associated to the sets~$Y(C)$ in Proposition~\ref{prop:tailprob} we have
\[
w(n) \sim \log n\qquad\text{as $n\to\infty$}.
\]
\end{prop}

\begin{proof}
From Proposition~\ref{prop:tailprob} we obtain
\[
w(n) = \sum_{j=1}^n \mu(\varphi_{Y(C)} > j) \sim \sum_{j=1}^n \frac1j \qquad\text{as $n\to\infty$}.
\]
Observing that
\[
 \sum_{j=1}^n \frac1j \sim \log n \qquad\text{as $n\to\infty$}
\]
completes the proof.
\end{proof}

\section{Proof of Theorem~\ref{thm:dyn_intervals}}
\label{sec:proofthm11}
Throughout this proof section we consider odd $q\geq 5$. In particular, $q\not=3$.
    We consider all measures as operators on the space of continuous functions.
    Let $[\alpha, \beta]$ be any compact subinterval of~$(0,1]$ and let $g\in C(X;\R)$ be any continuous function on~$X=[0,1]$.
    Then $g$ is Riemann-integrable as $X$ is compact and $g$ continuous. 
    The Farey map~$\FM$ is a topologically mixing AFN-map by Propositions~\ref{prop:Fareytopmix} and~\ref{prop:FareyAFN}.
    The set of indifferent fixed points of~$(X,\FM)$ is~$\{0\}$, and hence $X' = (0,1]$. See~\eqref{eq:complindfixed}.
    The tail probabilities of~$(X,\mc B, \FM,\mu)$ are regularly varying with exponent~$1$ and the tail probabilities summation functions are asymptotic to~$\log(n)$ by Propositions~\ref{prop:tailprob} and~\ref{prop:tailsum}.
    Hence \cite[Theorem~1.1]{melbourne_operator_2012} (see Theorem~\ref{thm:MT}) provides the convergence
    \[
    \log(n) \wh\FM^{n}\left(\frac{g}{h}\right)\quad \stackrel{n\to\infty}{\longrightarrow}\quad \int_X \frac{g}{h} \dd\mu
    \]
    and establishes that this convergence is uniform on~$[\alpha,\beta]$. 
    Applying the defining relation of the transfer operator~$\wh\FM$ from~\eqref{eq:defTO}, Lebesgue's dominated convergence theorem and the relation $\dd\mu=h\dd\leb$ (in this order), we obtain that
    \begin{align*}
        \lim_{n\to\infty} \log(n)  m\vert_{\FM^{-n}([\alpha,\beta])}(g)
        &= \lim_{n\to\infty}
		\int\limits_{\FM^{-n}([\alpha,\beta])} \log(n) g \dd\leb
		\\
		&=
		\lim_{n\to\infty} \int\limits_{[\alpha,\beta]} \log(n)\wh{\FM}^n\left(\frac{g}{h}\right) \dd\mu
		\\
		& =  \mu([\alpha,\beta]) \int_X \frac{g}{h} \dd\mu =  \mu([\alpha,\beta]) \int_X g \dd\leb 
		\\
		& =  \mu([\alpha,\beta]) \leb(g)\,.
	\end{align*}
	Observing that $\mu([\alpha,\beta]) = \log(\beta/\alpha)$ now completes the proof.
\qed

\section{Proof of Theorems~\ref{thm:jac_weight_lim_intro} and~\ref{thm:weighted_distr}}\label{sec:proofweight}

In this final section we provide detailed proofs for Theorems~\ref{thm:jac_weight_lim_intro} and~\ref{thm:weighted_distr}.
As mentioned in the Introduction (Section~\ref{sec:intro}), Theorem~\ref{thm:weighted_distr} is essentially only a specialization of Theorem~\ref{thm:jac_weight_lim_intro} to reduced fractions, for which reason we start by presenting this argumentation.
The rather longish proof of Theorem~\ref{thm:jac_weight_lim_intro} follows further below. Throughout this section we consider all odd $q\geq 3$.

\begin{lemma}
	The cusp points $\Gamma.\infty$ and the reduced fractions (understood as equivalence class) in $\Z[\cw]\times \Z[\cw]$ are in bijection.
\end{lemma}
\begin{proof}
	Clearly, each point in $\Gamma.\infty$ can be represented by a reduced fraction.
	We show that the representation is unique.
	Let $(a,c)$ be reduced and let $g=\textbmat{a}{\ast}{c}{\ast}$. Suppose that $h\in\Gamma$ is such that $h.\infty = a/c = g.\infty$. Then $g^{-1}h.\infty = \infty$ which means that $g^{-1}h$ is an element of the stabilizer $\Stab_{\Gamma}(\infty)$ of $\infty$ in $\Gamma$.
	The generating element $T=\begin{bsmallmatrix}1 & \cw\\ 0 & 1\end{bsmallmatrix}$ also stabilizes $\infty$.
	All stabilizers in a Fuchsian group are cyclic, and hence the Poincaré Theorem on fundamental polyhedrons yields that
	\begin{equation*}
		\Stab_{\Gamma}(\infty) = \langle T\rangle\,.
	\end{equation*}
	Therefore, there exists some $n\in\N$ such that
	\begin{equation*}
		h = gT^n =
		\begin{bmatrix}
		a & \ast\\
		c & \ast
		\end{bmatrix}
		\begin{bmatrix}
			1 & n\cw\\
			0 & 1
		\end{bmatrix}
		= \begin{bmatrix}
			a & \ast \\
			c & \ast
		\end{bmatrix}\,.
	\end{equation*}
	Hence, the reduced fraction representing $a/c$ is unique.
\end{proof}

\begin{lemma}
\label{lem:Qredfrac}
	The pair $(a,c)\in\Z[\cw]\times \Z[\cw]$ is a reduced fraction if and only if there exists $g\in \langle\Gamma, Q\rangle$ such that
	\begin{equation*}
		g = \begin{bmatrix}
			a & \ast\\ c & \ast
		\end{bmatrix}.
	\end{equation*}
\end{lemma}
\begin{proof}
	Since $Q$ is an outer symmetry, we have
	\begin{equation*}
		\langle\Gamma, Q\rangle = \Gamma \cup Q\Gamma\,.
	\end{equation*}
	In particular, if $(a,c)$ is reduced then there is an element in $\Gamma\subseteq\langle\Gamma, Q\rangle$ with a matrix representation of the form~$\textbmat{a}{\ast}{c}{\ast}$.
	For the other direction, let $g=\textbmat{a}{b}{c}{d}\in \langle\Gamma, Q\rangle$. 
	If $g\in\Gamma$, then $(a,c)$ is reduced by definition.
	Suppose that $g = Qh$ for some $h\in\Gamma$.
	Then $gQ\in\Gamma$, and for the generator~$S$ of $\Gamma$ as considered in Section~\ref{sec:mainresults} it holds
	\begin{equation*}
		gQS = \begin{bmatrix}	a & -b\\ c & -d	\end{bmatrix} \in \Gamma\,,
	\end{equation*}
	which yields that $(a,c)$ is reduced.
\end{proof}

The key observation needed for establishing Theorem~\ref{thm:weighted_distr} (in addition to the validity of Theorem~\ref{thm:jac_weight_lim_intro}) is a (natural) correspondence between words of length~$n$ (i.e., elements of~$W_n$) and reduced fractions of level~$n$ related to a given reduced fraction $v/w\in\Gamma.\infty \cap (0,1)$ (i.e., elements of~$\RF_n(v,w)$), as stated in the following lemma.

\begin{lemma}\label{lem:WnTx}
For each reduced fraction~$v/w\in \Gamma.\infty \cap (0,1)$ and each~$n\in\N$, the sets~$W_n$ and~$\RF_n(v,w)$ are in bijection via the map
\begin{equation}
\label{eq:mapWRF}
 W_n \to \RF_n(v,w)\,,\quad  \bmat{a}{b}{c}{d} \mapsto (av+bw, cv+dw)\,.
\end{equation}
For $v/w=1$, the map in~\eqref{eq:mapWRF} is $2:1$, i.e., each element in~$\RF_n(v,w)$ arises from exactly two elements in~$W_n$.
\end{lemma}

We emphasize that under the identification of any reduced fraction $(v,w)$ with the point $v/w$ in $P^1(\R)$, the map in Lemma~\ref{lem:WnTx} reads
\[
 W_n \to \RF_n(v,w)\,,\quad h\mapsto h.\frac{v}{w}\,.
\]
Indeed, the presentation of~$h.(v/w)$ for $h\in W_n$, $h = \textbmat{a}{b}{c}{d}$, as a reduced fraction is $(av+bw, cv+dw)$.

\begin{proof}[Proof of Lemma~\ref{lem:WnTx}]
Let $v/w\in\Gamma.\infty\cap (0,1]$.
We show that the map in~\eqref{eq:mapWRF} is well-defined.
For each $h\in W_n$ given by $h=\begin{bsmallmatrix}a & b\\ c & d\end{bsmallmatrix}$ we have
\begin{equation*}
	h.\frac{v}{w} = \frac{a(v/w)+b}{c(v/w)+d} = \frac{av + bw}{cv+dw}\,,
\end{equation*}
which is an element in $\FM^{-n}(v/w)$.
Since $(v,w)$ is a reduced fraction, there exists $\begin{bsmallmatrix}v & *\\ w & *\end{bsmallmatrix}\in\Gamma$ and therefore
\begin{equation*}
	\bmat{a}{b}{c}{d}\bmat{v}{*}{w}{*} = \bmat{av+bw}{*}{cv+dw}{*}\in\langle\Gamma,Q\rangle\,.
\end{equation*}
Taking advantage of Lemma~\ref{lem:Qredfrac} in case that $\det h = -1$, this shows that the pair $(av+bw, cv+dw)$ is the reduced fraction representing $h.(v/w)$ and that it is contained in $\RF_n(v,w)$.

For $v/w\not=1$, let $h_1=\begin{bsmallmatrix}a_1 & b_1\\ c_1 & d_1\end{bsmallmatrix}, h_2=\begin{bsmallmatrix}a_2 & b_2\\ c_2 & d_2\end{bsmallmatrix}\in W_n$ be such that
\begin{equation*}
	(a_1v+b_1w, c_1v+d_1w) = (a_2v+b_2w, c_2v+d_2w)\,.
\end{equation*}
Then $a_1v+b_1w=a_2v+b_2w$ and $c_1v+d_1w=c_2v+d_2w$.
Since $(v,w)\neq (0,0)$ and $v\neq w$ it follows that $h_1=h_2$, hence the map in~\eqref{eq:mapWRF} is injective.

For $v/w=1$, as $Q.1=1$, we have 
\[
 g_k^{-1}.1 = (Qg_k)^{-1}.1 \qquad\text{for $k\in\left\{ \frac{q+1}2,\ldots, q-1 \right\}$}
\]
and $g_k^{-1}.1\not=g_\ell^{-1}.1$ for $k,\ell\in\{(q+1)/2,\ldots, q-1\}$, $k\not=\ell$. Therefore, the map from $\Lambda$ to $\Lambda.1$ is $2:1$. As any point in $\Lambda.1$ is different from~$1$ and $W_n = W_{n-1}\Lambda$, the argument from above for $v/w\not=1$ implies that the map in~\eqref{eq:mapWRF} is indeed $2:1$ for $v/w=1$.

For all $v/w\in\Gamma.\infty\cap (0,1]$, surjectivity of the map follows by definition of $\RF_n(v,w)$.
\end{proof}

With these preparations we can now provide a short proof of Theorem~\ref{thm:weighted_distr}.

\begin{proof}[Proof of Theorem~\ref{thm:weighted_distr}]
Using Theorem~\ref{thm:jac_weight_lim_intro} for $x=v/w$ and combining with the map from Lemma~\ref{lem:WnTx} we obtain
\[
 m = \starlim_{n\to\infty} \frac{v}{w} \log(n) \sum_{h\in W_n} \left|h'\left(\frac{v}{w}\right)\right|\delta_{h.(v/w)} = \starlim_{n\to\infty} c_{v/w} vw\log(n) \sum_{(r,s)\in \RF_n(v,w)} \frac{1}{s^2} \delta_{r/s}\,,
\]
where $c_1=2$ and $c_x=1$ for $x\not=1$.
For the second equality, we use that for each $h\in W_n$ we have
\[
 h'\left(\frac{v}{w}\right) = \frac{w^2}{s^2}
\]
with $r/s$ denoting be the reduced form of~$h.(v/w)$.
\end{proof}

We now turn to the discussion of Theorem~\ref{thm:jac_weight_lim_intro} and start with some preparations.
As we will establish Theorem~\ref{thm:jac_weight_lim_intro} first for $x\not=1$ and then use this result to show it for $x=1$, we first let $x\in (0,1)$.
We emphasize that almost all objects defined in what follows depend on~$x$.
Anyhow, this dependence will not be reflected in the notation in favor of simplicity.
For~$n\in\N$ set 
\[
\varrho_{n} \coloneqq x \log(n) \sum_{h\in W_n} |h'(x)|\delta_{h.x}\,,
\]
which is a Borel measure on~$(0,1]$. 
Theorem~\ref{thm:jac_weight_lim_intro} is the statement that
\begin{equation}\label{eq:jacweightabstract}
\starlim_{n\to\infty} \varrho_{n} = \leb\,.
\end{equation}
In order to establish~\eqref{eq:jacweightabstract}, we will take advantage of  
Theorem~\ref{thm:dyn_intervals}.

For~$\eps>0$ we set 
\[
V_\eps \coloneqq V_\eps(x) \coloneqq \left(x-\frac{\eps}{2}, x+ 
\frac{\eps}{2}\right)\,,
\]
the open $\eps/2$-neighborhood of~$x$ in~$\R$. For $n\in\N_0$ we further set 
\[
\mc V_{\eps,n} \coloneqq \FM^{-n}\bigl( V_\eps \bigr)
\]
and 
\[
\leb_{\eps,n} \coloneqq \frac{\log(n)}{\mu\bigl( V_\eps \bigr)}\,\leb\vert_{\mc V_{\eps,n}}\,,
\]
which is a Borel measure on~$(0,1]$. By Theorem~\ref{thm:dyn_intervals} we
know that 
\begin{equation}\label{eq:thmleb}
\starlim_{n\to\infty} \leb_{\eps,n}  = \leb\,.
\end{equation}
In what follows, we will show that, for $\eps>0$ sufficiently small, the measures~$\leb_{\eps,n}$ and $\varrho_n$
are sufficiently near to each other such that their weak-star limits (as $n\to\infty$ and $\eps\to 0$, in this order) are the
same. For this, we will proceed via their distribution functions. We let
$\Delta^\leb_{\eps,n}$ denote the distribution function of~$\leb_{\eps,n}$, and 
$\Delta^\varrho_n$ the distribution function of~$\varrho_n$. That is, for all 
$y\in (0,1]$, 
\[
\Delta^\leb_{\eps,n}(y) \coloneqq \leb_{\eps,n}\bigl( (0,y] 
\bigr)\qquad\text{and}\qquad \Delta^\varrho_n(y) \coloneqq \varrho_n\bigl( (0,y] 
\bigr)\,.
\]
From now on we suppose that $\eps>0$ is sufficiently small such that 
$V_\eps\subseteq (0,1)$. Then, for all $n\in\N$, 
\[
 \FM^{-n}\left( V_\eps \right) = \bigcup_{h\in W_n} h.V_\eps
\]
is a disjoint union due to the pairwise disjointness of the sets $h.(0,1)$,
$h\in W_n$. In particular,
\[
 \leb\left( \FM^{-n}\left( V_\eps \right) \right) = \sum_{h\in W_n} 
\leb\left( h.V_\eps \right)\,.
\]
For each~$y\in (0,1]$ we now obtain that
\begin{align}
\nonumber
\bigl| \Delta^\leb_{\eps,n}(y) & - \Delta^\varrho_n(y) \bigr|
\\
\nonumber
& = \log(n)\, 
\left| 
\frac{1}{\mu\bigl( V_\eps \bigr)} m\left(\FM^{-n}\bigl( V_\eps \bigr)\cap  
(0,y] 
\right) -  x \sum_{h\in W_n} |h'(x)|\,\delta_{h.x}\bigl( (0,y] \bigr)\right|
\\
\nonumber
& = \log(n) \left| \sum_{h\in W_n} \frac{1}{\mu\bigl( V_\eps \bigr)} \leb\bigl( 
h.V_\eps \cap (0,y] \bigr) - x|h'(x)|\,\delta_{h.x}\bigl( (0,y] \bigr)  \right|
\\
\label{eq:bounddiffdistr}
& \leq \log(n) \sum_{h\in W_n} \left| \frac{1}{\mu\bigl( V_\eps \bigr)} 
\leb\bigl( h.V_\eps \cap (0,y] \bigr) - x|h'(x)|\,\delta_{h.x}\bigl( (0,y] 
\bigr) \right|\,.
\end{align}
For each $h\in W_n$, the set $h.V_\eps$ is an open interval and hence we either have
$h.V_\eps \subseteq (0,y]$ or $h.V_\eps \cap (0,y] = \emptyset$ or $y\in h.V_\eps$.
For the former two situations we immediately obtain the following simplification of the summand for~$h$ in~\eqref{eq:bounddiffdistr}:
\begin{align*}
& \left| \frac{1}{\mu\bigl( V_\eps \bigr)}
\leb\bigl( h.V_\eps \cap (0,y] \bigr) - x|h'(x)|\,\delta_{h.x}\bigl( (0,y]
\bigr) \right|
\\
& \hphantom{\frac{1}{\mu\bigl( V_\eps \bigr)}
\leb\bigl( h.V_\eps \cap (0,y] \bigr) -}
=
\begin{cases}
\left| \frac{\leb( h.V_\eps)}{\mu( V_\eps)}
 - x|h'(x)| \right| & \text{if $h.V_\eps \subseteq (0,y]$,}
\\[1mm]
0 & \text{if $h.V_\eps \cap (0,y] = \emptyset$.}
\end{cases}
\end{align*}
By disjointness of the sets $h.V_\eps$, $h\in W_n$, the situation $y\in h.V_\eps$ can materialize for at most one $h\in W_n$, in which case we denote this element by $h_{y,\eps,n}$. Combining these arguments with~\eqref{eq:bounddiffdistr}, we obtain the following, \textit{a priori} rather rough estimate for the difference of the two distribution
functions at~$y$:
\begin{align}
\nonumber
&\bigl| \Delta^\leb_{\eps,n}(y) - \Delta^\varrho_n(y) \bigr|
\\
\nonumber
& \hphantom{\Delta^\leb_{\eps,n}(y) -}
\leq  \log(n)\left| \frac{1}{\mu\bigl( V_\eps \bigr)}
\leb\bigl( h_{y,\eps,n}.V_\eps \cap (0,y] \bigr) - x|h_{y,\eps,n}'(x)|\,\delta_{h_{y,\eps,n}.x}\bigl( (0,y] \bigr) \right|
\\
\label{eq:estimsum}
& \hphantom{\Delta^\leb_{\eps,n}(y) -  \leq\ }
+ \log(n) \sum_{h\in W_n} \left| \frac{1}{\mu\bigl( V_\eps \bigr)}
\leb\bigl( h.V_\eps\bigr) - x|h'(x)| \right|\,,
\end{align}
where the summand with $h_{y,\eps,n}$ is to be understood as~$0$ if $h_{y,\eps,n}$ does not exist.

In what follows we will prove that the summand with $h_{y,\eps,n}$ as well as the sum over~$W_n$ becomes
small as $n\to\infty$ and $\eps\to 0$. Before we start with a rigorous discussion, we provide an indication why this might be expected. To that end we note that
\begin{equation}\label{eq:hVepsinterval}
h.V_\eps =
\begin{cases}
\left( h.\left(x-\frac{\eps}{2}\right), h.\left( x+\frac{\eps}{2}\right) 
\right) 
& \text{if $\det h > 0$}
\\
\left( h.\left(x+\frac{\eps}{2}\right), h.\left( x-\frac{\eps}{2}\right) 
\right) 
& \text{if $\det h < 0$\,.}
\end{cases}
\end{equation}
Thus, 
\[
\leb\bigl( h.V_\eps \bigr) = \left| h.\left(x-\frac{\eps}2\right) - h.\left(x+ 
\frac{\eps}{2}\right)\right|
\]
and hence
\begin{equation*}
\lim_{\eps \to 0} \frac{ \leb\bigl( h.V_\eps \bigr) }{\eps} =  \lim_{\eps\to 0} 
\frac{ \left| h.\left(x-\frac{\eps}2\right) - h.\left(x+ 
\frac{\eps}{2}\right)\right| }{\eps} =  |h'(x)|\,.
\end{equation*}
(We note that $h$ is indeed differentiable on all of~$(0,1)$, for each $h\in W_n$.)
Further, 
\begin{equation}\label{eq:munotzero}
\mu\bigl( V_\eps \bigr) = \int\limits_{x-\eps/2}^{x+\eps/2} \frac1x\,d\leb = 
\log\left( \frac{x+\eps/2}{x-\eps/2}\right)
\end{equation}
and hence
\begin{equation}\label{eq:lim_muUeps}
\lim_{\eps\to 0} \frac{\eps}{\mu\bigl( V_\eps \bigr)} = \lim_{\eps\to 0} 
\frac{\eps}{ \log\left( \frac{x+\eps/2}{x-\eps/2}\right) } = x\,.
\end{equation}
Thus,
\begin{equation*}
\lim_{\eps\to 0} \frac{1}{\mu\bigl(V_\eps\bigr)}\, \leb\bigl( h.V_\eps \bigr) = 
\lim_{\eps \to 0} \frac{\eps}{\mu\bigl(V_\eps\bigr)}\, \frac{ \leb\bigl( 
h.V_\eps \bigr)}{\eps} = x|h'(x)|\,.
\end{equation*}
This shows that each summand in the $W_n$-sum in~\eqref{eq:estimsum} converges to~$0$ as $\eps\to 0$. However, for showing that the same is true for the full sum in the limit $n\to\infty$, a more subtle argumentation is necessary. The consideration above further gives an indication that the $h_{y,\eps,n}$-summand becomes small for $\eps\to 0$ and, in a certain sense, also for $n\to\infty$.

In the remainder of this section, we complete these heuristics with rigorous proofs. We start with two auxiliary results in Lemmas~\ref{lem:wordorder} and~\ref{lem:abs-h} below, finding the largest derivative among the elements in~$W_n$.

\begin{lemma}\label{lem:wordorder}
	For every $h\in W_1$ and $x\in[0,1]$ we have $g_{q-1}^{-1}.x\leq h.x$.
\end{lemma}
\begin{proof}
	We recall the matrix representation of~$g_k$ in~\eqref{eq:gk}.
	Since $s(k)\leq s(k-1)$ for every $k\in\{(q+1)/2,\dotsc,q-1\}$, we see that for $(q+1)/2\leq k_1<k_2\leq q-1$ and $x\in[0,1]$ it holds
	\begin{equation*}
		(s(k_2)x+s(k_2+1))(s(k_1-1)x+s(k_1)) \leq (s(k_2-1)x + s(k_2))(s(k_1)x + s(k_1+1))
	\end{equation*}
	which is equivalent to
	\begin{equation*}
		g_{k_2}^{-1}.x\leq g_{k_1}^{-1}.x\quad\text{for all $x\in[0,1]$}\,,
	\end{equation*}
	and in particular $g_{q-1}^{-1}.x\leq g_{k-1}^{-1}.x$.
	Since $g_k^{-1}$ for any $k\in\{(q+1)/2,\dotsc,q-1\}$ is monotone increasing and $Q.[0,1]=[1,\infty]$, it follows that $g_k^{-1}.x\leq g_k^{-1}Q.x$ for all $x\in[0,1]$.
\end{proof}

For convenience in the proof of the following lemma, we note that for each $k\in\Z$,
\begin{equation}\label{eq:Qgk}
 Qg_k = \bmat{ -s(k-1) }{ s(k) }{ s(k) }{ -s(k+1) }\,.
\end{equation}

\begin{lemma}\label{lem:abs-h}
	For each $n\in\N$ and $h\in W_n$, $\lvert h'\rvert$ is monotone
decreasing.
	Further, $\left((g_{q-1}^{-1})^n\right)' \geq \lvert h'\rvert$.
\end{lemma}

\begin{proof}
From~\eqref{eq:gk} and \eqref{eq:Qgk} we obtain that each element in the alphabet~$\Lambda=W_1$ has
a matrix representation with purely non-negative coefficients of which the two bottom-row coefficients are positive and the two top-row coefficients cannot vanish simultaneously.
Thus, also $h\in W_n$ has such a matrix representation as being a composition of elements in~$W_1$, say
\[
h = \bmat{a}{b}{c}{d}
\]
with $c,d>0$ (and $a,b\geq 0$). Since $|h'(\xi)| = (c\xi+d)^{-2}$, we immediately see that $|h'|$ is monotone decreasing.

It remains to establish the second statement of the lemma. 
For each $k\in\{(q+1)/2,\dotsc,q-1\}$ we obtain from~\eqref{eq:gk} and~\eqref{eq:Qgk} that
\begin{align*}
\lvert (g_k^{-1})'(x)\rvert = \frac{1}{\bigl(s(k-1)x+s(k)\bigr)^2},
\shortintertext{and}
\lvert (g_k^{-1}Q)'(x)\rvert = \frac{1}{\bigl(s(k)x+s(k-1)\bigr)^2}
\end{align*}
for all $x\in [0,1]$.
For every $k\in\{(q+1)/2,\dotsc,q-1\}$ we have $s(k)\leq s(k-1)$, and hence, for all $x\in[0,1]$,
\[
s(k)x + s(k-1) \geq s(k-1)x + s(k) > 0\,.
\]
Thus, for $x\in [0,1]$,
	\begin{equation*}
		\lvert (g_k^{-1})'(x)\rvert \geq \lvert (g_k^{-1}Q)'(x)\rvert\,.
	\end{equation*}
Therefore the element in $W_1$ with the largest absolute derivative is of the form $g_k$ for
$k\in\{(q+1)/2,\dotsc,q-1\}$.
Using again $s(k)\leq s(k-1)$, we deduce
	\begin{equation*}
		s(k-1)x+s(k) \geq s(q-2)x+s(q-1) = \cw x+1
	\end{equation*}
for $k\in \{(q+1)/2,\dotsc,q-1\}$ and $x\in [0,1]$.
It follows that for $h\in W_1$,
\begin{equation}\label{eq:derivsmall}
\lvert h'\rvert \leq \lvert(g_{q-1}^{-1})'\rvert
\end{equation}
on~$[0,1]$.
Further, for $h\in W_1$ and $x\in [0,1]$ we have
\begin{equation}\label{eq:pointsmall}
g_{q-1}^{-1}.x \leq h.x\,.
\end{equation}
See Lemma~\ref{lem:wordorder}.
For any $h= h_1\dots h_n\in W_n$, combining~\eqref{eq:derivsmall}, \eqref{eq:pointsmall} and the fact that $|h'|$ is decreasing, the product rule for derivatives yields
\[
|h'| \leq \left|\left((g_{q-1}^{-1})^{n}\right)'\right|
\]
on all of~$[0,1]$. Finally, because
\[
g_{q-1}^{-n} = \bmat{1}{0}{n\cw}{1}
\]
and hence
\[
 \left(g_{q-1}^{-n}\right)'(x) = \frac{1}{(n\cw x + 1)^2},
\]
the derivative of $g_{q-1}^{-n}$ is positive on $[0,1]$. This completes the proof.
\end{proof}

We can now establish estimates for the $h_{y,\eps,n}$-summand and the summands of the $W_n$-sum in~\eqref{eq:estimsum} that are sufficient for our purposes.

\begin{lemma}\label{lem:summands}
We have the following estimates:
\begin{enumerate}[label=$\mathrm{(\roman*)}$, ref=$\mathrm{\roman*}$]
\item\label{sum:genh} For all $x\in (0,1)$, all $\eps>0$ such that $V_\eps(x)\subseteq (0,1)$, all $n\in\N$ and all $h\in W_n$ we have
\[
\left| \frac{1}{\mu\bigl(V_\eps\bigr)}\leb\bigl(h.V_\eps\bigr) - x|h'(x)|  \right| \leq \left( \left| \frac{\eps}{\mu\bigl(V_\eps\bigr)} - x \right| +  \frac{\eps}{x}  \right) \frac{\leb\bigl(h.V_\eps\bigr)}{\eps}\,,
\]
where $V_\eps\coloneq V_\eps(x)$.
\item\label{sum:specialh} Let $x\in (0,1)$ and $\eps>0$ such that $V_\eps(x)\subseteq (0,1)$. Then there exists $M>0$ (depending on $x$ and~$\eps$) such that for all $y\in (0,1]$ and all $n\in\N$ we have
\[
\left| \frac{1}{\mu\bigl(V_\eps\bigr)} \leb\bigl( h_{y,\eps,n}.V_\eps \cap (0,y]\bigr) - x|h_{y,\eps,n}'(x)| \delta_{h_{y,\eps,n}.x}\bigl( (0,y]\bigr)  \right| \leq \frac{M}{n^2}\,,
\]
where $V_\eps\coloneqq V_\eps(x)$.
\end{enumerate}
\end{lemma}

\begin{proof}
Let $x\in (0,1)$, $\eps>0$ such that $V_\eps=V_\eps(x)\subseteq (0,1)$, $n\in\N$ and $h\in W_n$.
In order to establish \eqref{sum:genh}, we first show that
\begin{equation}\label{eq:genh_intest}
 \left| |h'(x)|\,\frac{\eps}{\leb(h.V_\eps)} - 1 \right| \leq \frac{\eps}{x^2}\,.
\end{equation}
To that end we note that
\[
m(h.V_\eps) = \left| h.\left(x-\frac{\eps}2\right) - h.\left(x+\frac{\eps}2 \right) \right| = \eps |h'(\xi)|
\]
for some $\xi\in V_\eps$ (depending on~$x$), by the mean value theorem. We pick the representative of~$h$ with positive bottom-row entries, say
\[
  h = \bmat{a}{b}{c}{d}
\]
with $c,d>0$. Then
\begin{align}\label{eq:genh_step1}
|h'(x)|\,\frac{\eps}{\leb(h.V_\eps)} - 1 & = \frac{|h'(x)|}{|h'(\xi)|} - 1 = \frac{(c\xi+d)^2}{(cx+d)^2} - 1
 = \frac{c^2(\xi+x) + 2cd}{(cx+d)^2} (\xi-x)\,.
\end{align}
From $x,\xi\in (0,1)$ we obtain $0<x+\xi<2$, and $\xi\in V_\eps(x)$ implies $|\xi-x|<\eps/2$. Using these estimates in~\eqref{eq:genh_step1} yields
\begin{align*}
 \left| |h'(x)|\,\frac{\eps}{\leb(h.V_\eps)} - 1 \right| & \leq \eps\,\frac{c^2 + cd}{(cx+d)^2}
 = \frac{\eps}{x^2}\,\frac{c^2+cd}{\left( c + \frac{d}{x} \right)^2}
  < \frac{\eps}{x^2}\, \frac{c^2+cd}{(c+d)^2}  < \frac{\eps}{x^2}\,,
\end{align*}
which shows~\eqref{eq:genh_intest}. For the penultimate inequality we used that $d>0$ and $x\in (0,1)$ and hence $d/x > d$, and further $c>0$.
We conclude further that
\begin{align*}
\left| \frac{1}{\mu(V_\eps)}\leb(h.V_\eps) - x|h'(x)|\right| & = \left| \frac{\eps}{\mu(V_\eps)} \frac{\leb(h.V_\eps)}{\eps} - x|h'(x)| \right|
\\
& = \left| \left( \frac{\eps}{\mu(V_\eps)} - x \right)\frac{\leb(h.V_\eps)}{\eps} + x\left( \frac{\leb(h.V_\eps)}{\eps} - |h'(x)|  \right) \right|
\\
& \leq \left| \frac{\eps}{\mu(V_\eps)} - x \right| \frac{\leb(h.V_\eps)}{\eps}  + x \left| \frac{\leb(h.V_\eps)}{\eps} - |h'(x)| \right|
\\
& \leq \left( \left| \frac{\eps}{\mu(V_\eps)} - x\right| + \frac{\eps}{x}  \right) \frac{\leb(h.V_\eps)}{\eps}\,,
\end{align*}
where we used \eqref{eq:genh_intest} for the last inequality. This establishes~\eqref{sum:genh}.

In order to show~\eqref{sum:specialh} let $y\in (0,1]$ and set $h\coloneqq h_{y,\eps,n}$. We note that in this case $y\in h.V_\eps$.
Then
\begin{align}
\nonumber
\Big| \frac{1}{\mu\bigl( V_\eps \bigr)} \leb\bigl( h.V_\eps \cap (0,y] \bigr) &
- x|h'(x)|\delta_{h.x}\bigl( (0,y] \bigr) \Big|
\\
\nonumber
& \quad \leq \frac{1}{\mu\bigl( V_\eps \bigr)} \leb\bigl( h.V_\eps \cap (0,y]
\bigr) + x |h'(x)| \delta_{h.x}\bigl( (0,y] \bigr)
\\
\label{eq:yinhU}
& \quad \leq \frac{1}{\mu\bigl( V_\eps \bigr)} \leb\bigl( h.V_\eps \bigr) + x
|h'(x)|\,.
\end{align}
By Lemma~\ref{lem:abs-h}, $|h'|$ is monotone decreasing. Thus,
\begin{equation}\label{eq:boundhprime}
|h'(x)| \leq \left| h'\left( x - \frac{\eps}{2}\right)\right|
\end{equation}
and further, applying the mean value theorem,
\begin{align}\label{eq:bounddiam}
\leb\bigl( h.V_\eps \bigr) = \left| h.\left( x-\frac{\eps}{2}\right) - h.\left(
x+\frac{\eps}{2}\right) \right| \leq \eps \left|h'\left(x-\frac{\eps}{2}\right) \right|\,.
\end{align}
Using~\eqref{eq:boundhprime} and~\eqref{eq:bounddiam} in~\eqref{eq:yinhU}, we obtain that
\begin{equation}\label{eq:boundinhprime}
\left| \frac{1}{\mu\bigl( V_\eps \bigr)} \leb\bigl( h.V_\eps \cap (0,y] \bigr) -
x|h'(x)|\delta_{h.x}\bigl( (0,y] \bigr) \right| \leq \left(
\frac{\eps}{\mu\bigl( V_\eps \bigr)} + x \right) \left| h'\left( x -
\frac{\eps}2\right) \right|\,.
\end{equation}
By~\eqref{eq:munotzero} and~\eqref{eq:lim_muUeps}, the first factor of this bound is near~$2x$ for the whole range of admissible values for~$\eps$. I.e., there exists $C>0$ (depending on~$x$) such that
\[
\frac{\eps}{\mu\bigl( V_\eps \bigr)} + x \leq (C+1)x \leq C+1
\]
for all $\eps\in (0,\min\{2x,2-2x\})$. (The bounds on~$\eps$ deduce from the hypothesis $V_\eps(x)\subseteq (0,1)$.)
Regarding the second factor of the bound in~\eqref{eq:boundinhprime}, we recall from Lemma~\ref{lem:abs-h} that
\[
g_{q-1}^{-n} = \bmat{1}{0}{n\cw}{1}
\]
is the element of~$W_n$ with largest absolute derivative.
Therefore
\[
\left| h'\left( x - \frac{\eps}2\right) \right| \leq
\left(g_{q-1}^{-n}\right)'\left(x-\frac{\eps}{2}\right) = \frac{1}{\left( n\cw
\left(x-\frac{\eps}{2}\right) + 1 \right)^{2}} \leq
\frac{1}{\cw^2\left(x-\frac{\eps}{2}\right)^2 n^2}\,.
\]
Using the latter two estimates in~\eqref{eq:boundinhprime} we obtain
\[
\left| \frac{1}{\mu\bigl( V_\eps \bigr)} \leb\bigl( h.V_\eps \cap (0,y] \bigr) -
x|h'(x)|\delta_{h.x}\bigl( (0,y] \bigr) \right| \leq \frac{M}{n^2}
\]
with
\[
M \coloneqq \frac{C+1}{\cw^2\left(x-\frac{\eps}{2}\right)^2}\,.
\]
We note that $x-\eps/2\not=0$ by the choice of~$\eps$. We further note that $M$ is indeed independent of~$y$ and~$n$.
\end{proof}

With these preparations we can now provide a proof of Theorem~\ref{thm:jac_weight_lim_intro}, for which it now suffices to show that the difference of the distribution functions in~\eqref{eq:estimsum} vanishes as $n\to\infty$ and $\eps\to 0$ (in this order).

\begin{proof}[Proof of Theorem~\ref{thm:jac_weight_lim_intro}]
We first prove Theorem~\ref{thm:jac_weight_lim_intro} for $x\in (0,1)$.
To that end, we resume the notation from above and continue the estimate in~\eqref{eq:estimsum} by combining it with the bounds from Lemma~\ref{lem:summands}, indicating the dependency of $M$ on~$\eps$ as $M_\eps$.
We have
\begin{align*}
& \bigl| \Delta_{\eps,n}^{\leb}(y) - \Delta_n^\varrho(y)\bigr|
\\
& \hphantom{\Delta^\leb_{\eps,n}(y) -}
\leq \left( \left| \frac{\eps}{\mu(V_\eps)} - x\right| + \frac{\eps}{x}
\right)\frac{\mu(V_\eps)}{\eps}\,\frac{\log n}{\mu(V_\eps)}\sum_{h\in W_n}
\leb(h.V_\eps) + \frac{\log n}{n^2}M_\eps
\\
& \hphantom{\Delta^\leb_{\eps,n}(y) -}
= \left( \left| \frac{\eps}{\mu(V_\eps)} - x\right| + \frac{\eps}{x}
\right)\frac{\mu(V_\eps)}{\eps} \leb_{\eps,n}\bigl( (0,1]\bigr) + \frac{\log
n}{n^2}M_\eps\,.
\end{align*}
As $n\to\infty$, this bound converges to
\[
 \left( \left| \frac{\eps}{\mu(V_\eps)} - x\right| + \frac{\eps}{x} \right)
\leb\bigl( (0,1]\bigr) \frac{\mu(V_\eps)}{\eps} = \left( \left| 
\frac{\eps}{\mu(V_\eps)} - x\right| + \frac{\eps}{x}
\right)\frac{\mu(V_\eps)}{\eps}\,.
\]
due to the weak convergence of~$(\leb_{\eps,n})_n$ to~$\leb$. Thus, 
\[
\bigl| \Delta_{\eps,n}^{\leb}(y) - \Delta_n^\varrho(y)\bigr|
\leq \left( \left| \frac{\eps}{\mu(V_\eps)} - x\right| + \frac{\eps}{x}
\right)\frac{\mu(V_\eps)}{\eps}\,.
\]
For $\eps\to 0$, the term on the left hand side vanishes due to~\eqref{eq:lim_muUeps}. This completes the proof of Theorem~\ref{thm:jac_weight_lim_intro} for $x\in (0,1)$.

For the proof of Theorem~\ref{thm:jac_weight_lim_intro} for $x=1$, we note that for any $p\in\Lambda$, $p.1\not=1$, and that for all $n\in\N$, $W_{n+1} = W_n\Lambda$. We obtain 
\begin{align*}
\log(n+1)\sum_{h\in W_{n+1}} |h'(1)|\delta_{h.1} & = \log(n+1) \sum_{g\in W_n}\sum_{p\in\Lambda} |(gp)'(1)|\delta_{gp.1} 
\\
& = \log(n+1) \sum_{p\in\Lambda} |p'(1)| \sum_{g\in W_n} |g'(p.1)|\delta_{g.(p.1)}
\\
& = \sum_{p\in\Lambda} |p'(1)| \frac{1}{p.1}\, p.1\cdot \log(n+1)\sum_{g\in W_n} |g'(p.1)|\delta_{g.(p.1)}
\\
& \stackrel{n\to\infty}{\longrightarrow} \sum_{p\in\Lambda} |p'(1)| \frac{1}{p.1}\, \leb
\end{align*}
by applying Theorem~\ref{thm:jac_weight_lim_intro} for $p.1\in (0,1)$, $p\in\Lambda$. Further 
\[
 \sum_{p\in\Lambda}|p'(1)|\frac{1}{p.1} = \PF h(1) = h(1) = 1
\]
by Proposition~\ref{prop:eigenfunctionh}. Thus, 
\[
 \log(n+1)\sum_{h\in W_{n+1}} |h'(1)|\delta_{h.1} \stackrel{n\to\infty}{\longrightarrow} \leb\,,
\]
which completes the proof.
\end{proof}

\renewcommand{\bibfont}{\normalfont\small}
\printbibliography

\setlength{\parindent}{0pt}

\end{document}